\documentclass[12pt]{article}
\usepackage[utf8]{inputenc}
\usepackage[usenames,dvipsnames,svgnames,table]{xcolor}
\usepackage[margin=1in]{geometry}
\usepackage{graphicx}
\usepackage{subcaption}
\usepackage{tcolorbox}
\usepackage{mathtools}
\usepackage{times}
\usepackage{sp}
\usepackage[numbers,sort&compress]{natbib}
\usepackage{thm-restate}
\usepackage{enumitem}
\usepackage{tabularx}
\usepackage{setspace}
\usepackage[]{authblk}

\usepackage{pgf,tikz,pgfplots}
\pgfplotsset{compat=1.15}
\usepackage{mathrsfs}
\usetikzlibrary{arrows}

\usepackage{quiver}

\usepackage{subcaption}

\usepackage[colorlinks=true, citecolor=orange, linkcolor=black]{hyperref}

\newcommand{\proj}{\operatorname{proj}}
\newcommand{\gap}{\mathtt{gap}}
\newcommand{\pfeas}{\mathtt{pfeas}}
\newcommand{\dfeas}{\mathtt{dfeas}}
\newcommand{\kkt}{\mathtt{kkt}}
\newcommand{\tol}{\mathsf{tol}}
\newcommand{\terminate}{\mathtt{terminate}}
\newcommand{\restart}{\mathtt{restart}}

\title{{\bf Parameter Tuning with Generalization Guarantees for GPU-Accelerated Linear Programming}}

\author[]{Siddharth Prasad}
\author[]{Dravyansh Sharma}
\affil[]{{\normalsize Toyota Technological Institute at Chicago \\ \textcolor{magenta}{\nolinkurl{{sprasad,dravy}@ttic.edu}}}}
\date{}

\begin{document}

\maketitle

\begin{abstract}
  Recent research has developed practical, parallelizable first-order methods for large scale linear programming, but performance is highly dependent on hyperparameter selection. We derive generalization guarantees for hyperparameter tuning within (cu)PDLP, a state-of-the-art first-order LP solver designed for modern hardware. First, we pin down the behavior of PDHG, the primal-dual hybrid gradient algorithm that underlies PDLP, as a function of its step size and primal weight, leading to linear sample complexity guarantees for learning those parameters. We then conduct a structural analysis of PDLP, which augments PDHG with several specialized techniques like preconditioning, adaptive step sizes, averaging, adaptive restarts, and smoothed primal weight updates. Our analysis captures the behavior of the solution trajectory as a function of the hyperparameters and leverages recent advances in data-driven algorithm design to obtain polynomial sample complexity guarantees for learning those hyperparameters. Finally, we conduct proof-of-concept experiments that demonstrate the need for data-driven PDLP parameter tuning. Our results showcase the versatility of the data-driven algorithm design toolkit for principled hyperparameter tuning within solver-grade implementations of complex modern optimization algorithms.
\end{abstract}

\section{Introduction}

Linear programming (LP) is one of the most fundamental paradigms of optimization and operations research. LP solvers---traditionally based on the simplex algorithm or interior point methods and run on CPUs---are an extremely mature technology. They serve innumerable applications across science and industry. LP furthermore plays a notable role as a subroutine in tree-search based solvers for {\em integer programming}, the canonical methodology for practical discrete optimization~\citep{clautiaux2025last}.

\citet{applegate2021practical} have recently developed a first-order LP solver based on the {\em primal-dual hybrid gradient (PDHG)} algorithm, which is a gradient descent-ascent algorithm with convergence guarantees, of~\citet{chambolle2011first}. The new algorithm, {\em primal-dual linear programming (PDLP)}~\citep{applegate2021practical}, augments PDHG with a suite of enhancements including preconditioning, adaptive restarts, adaptive step sizes, and primal weight updates~\citep{applegate2021practical, lu2025cupdlp}. The development of PDLP---which was the first demonstration of a practical first-order LP solver that is competitive with state-of-the-art solvers---has spurred an active resurgence in work on first-order methods for linear programming. Interest in this research area is especially bolstered by the fact that the compute-dominant operations of first-order methods are matrix-vector products, which makes PDLP well-suited for GPU acceleration (unlike simplex and barrier/interior point methods). In particular,~\citet{lu2025cupdlp} present cuPDLP, which is an implementation of PDLP designed to run on GPUs. Versions of PDLP and cuPDLP are now implemented in a number of commercial and open-source LP solvers such as Gurobi, Xpress, Google OR-Tools, NVIDIA cuOpt, HiGHS, Cardinal Optimizer, {\em etc.}\footnote{PDLP is not without shortcomings, a primary one being its performance degradation when high-accuracy solutions are demanded~\citep{lu2025overview}. Nevertheless, it has established itself as a fresh contender in the LP solver arena.}

In this paper we derive generalization guarantees for (cu)PDLP hyperparameter tuning, which is a key aspect that impacts performance~\citep{lu2025overview}. That is, we bound the error between the empirical performance over a training set of LP instances and expected test-time performance, uniformly over all hyperparameter settings. We thus obtain a {\em sample complexity} guarantee that bounds the size of a training set needed so that the difference between training performance and expected test-time performance is small (which in particular ensures that parameter selection---via an algorithm such as empirical risk minimization (ERM)---does not overfit to the training set). 
Our generalization theory for PDLP showcases recent advancements in {\em data-driven algorithm design}~\citep{balcan2020data} with {\em Pfaffian} structure~\citep{balcan2025algorithm}, which we use to pin down the structure of PDLP as a function of its hyperparameters.
We validate our theory through simple, proof-of-concept experiments that illustrate the need for data-dependent PDLP parameter tuning.\looseness-1

Our results, detailed further in the following subsection, are a testament to the data-driven algorithm design toolkit in its ability to yield principled theoretical guarantees for hyperparameter tuning in solver-grade implementations of mathematical optimization algorithms.

\paragraph{Our contributions}

In Section~\ref{sec:pdhg}, we pin down the behavior of vanilla PDHG~\citep{chambolle2011first} as a function of its step size $\eta$ and primal weight $\omega$. We show that the primal and dual iterates at each round are piecewise-rational functions of $(\eta,\omega)$, and bound the number of pieces and degrees, leading to polynomial generalization bounds.

Next, in Section~\ref{sec:pdlp}, we study the behavior of PDLP as a function of its {\em primal weight update smoothing parameter} $\theta$ and its {\em Pock-Chambolle preconditioning parameter} $\alpha$. The former parameter controls how primal weights are updated across restarts and the latter parameter controls a matrix norm in the preconditioning procedure~\citep{pock2011diagonal}. Open source implementations of PDLP~\citep{applegate2021practical,lu2025cupdlp,lu2024mpax,maaz2025pdlp} all hard-code $\theta = 0.5$ and $\alpha = 1$, though \citet{lu2025overview} point out that $\theta$ and $\alpha$ are critical hyperparameters for performance. (PDLP has other tunable parameters---in particular those controlling restart frequency---and our theory readily extends to those parameters as well. The parameters $\theta,\alpha$ that we focus on are the ones that most directly impact the actual primal-dual iterates.)

Our analysis of PDLP leverages recent advances in the theory of data-driven algorithm design. A key development in this line of work, due to~\citet{balcan2025algorithm}, is the ability to analyze algorithms that compute {\em Pfaffian functions}~\citep{khovanskiui1991fewnomials} of their hyperparameters. Pfaffian functions encompass polynomials (the most basic kind of Pfaffian function), exponentials, logarithms, {\em etc.} We show that the primal and dual iterates of PDLP are piecewise-Pfaffian functions, and bound the complexity of these functions in terms of the lengths of their corresponding {\em Pfaffian chains.} A primary conceptual contribution of the present paper is to showcase the power and flexibility of the recent theory by using it understand the behavior of a complex numerical optimization algorithm as a function of its hyperparameters.

Finally, in Section~\ref{sec:experiments}, we conduct simple, proof-of-concept experiments that demonstrate the importance of data-dependent PDLP parameter tuning. For various LP instance distributions, we illustrate the effects of (i) tuning $\theta$, keeping $\alpha$ fixed at its default value of $1$, (ii) tuning $\alpha$, keeping $\theta$ fixed at its default value of $0.5$, and (iii) tuning $\theta$ and $\alpha$ simultaneously. The results show that parameter tuning can have a dramatic impact on convergence rate and that parameter tuning ought to be done in a data-dependent way. That is, the optimal parameters can vary greatly across instance distributions.

\paragraph{Related work} In Appendix~\ref{app:related_work}, we survey additional related work on first-order methods for mathematical programming and data-driven algorithm design for optimization.

\section{Problem Formulation and Background}\label{sec:problem_formulation}

PDLP is a first-order method for solving linear programs (LPs) with primal and dual form
\begin{equation*}\label{eq:primal+dual}
\min_{\vec{x}\in\R^n}\left\{\vec{c}^\top\vec{x} :
\begin{aligned}
&\vec{G}\vec{x}\ge\vec{h}\\ &\vec{A}\vec{x}=\vec{b}\\ &\vec{\ell}\le\vec{x}\le\vec{u}
\end{aligned}\right\};~\max_{\vec{y}\in\R^{m_1+m_2}, \vec{\lambda}\in\R^n}\left\{\vec{q}^\top\vec{y} + \vec{\ell}^\top\vec{\lambda}^{+} -\vec{u}^\top\vec{\lambda}^{-} : 
\begin{aligned} &\vec{c}-\vec{K}^\top\vec{y}=\vec{\lambda} \\ &\vec{y}_{1:m_1}\ge\vec{0}\\ &\vec{\lambda}\in\Lambda_1\times\cdots\Lambda_n
\end{aligned}
\right\},
\end{equation*}
respectively, where $\vec{c}\in\R^n$, $\vec{G}\in\R^{m_1\times n}$, $\vec{h}\in\R^{m_1}$, $\vec{A}\in\R^{m_2\times n}$, $\vec{b}\in\R^{m_2}$, $\vec{\ell}\in(\R\cup\{-\infty\})^n$, $\vec{u}\in(\R\cup\{\infty\})^n$, $\vec{K}^\top = \begin{bmatrix}\vec{G}^\top & \vec{A}^\top\end{bmatrix}$, $\vec{q}^\top = \begin{bmatrix}\vec{h}^\top & \vec{b}^\top\end{bmatrix}$, and $\Lambda_i = \{0\}$ if $\ell_i = -\infty, u_i=\infty$; $\R_{\le 0}$ if $\ell_i=-\infty, u_i\in\R$; $\R_{\ge 0}$ if $\ell_i\in\R, u_i = \infty$; $\R$ if $\ell_i, u_i\in \R$. It builds upon the Primal-Dual Hybrid Gradient (PDHG) algorithm of~\citet{chambolle2011first}, which operates on the equivalent saddle point problem $\min_{\vec{x}\in X}\max_{\vec{y}\in Y} \vec{c}^\top\vec{x} - \vec{y}^\top \vec{K}\vec{x}+\vec{q}^\top\vec{y}$, where $X = \{\vec{x}\in\R^n : \vec{\ell}\le\vec{x}\le\vec{u}\}$ and $Y = \{\vec{y}\in\R^{m_1+m_2} : \vec{y}_{1:m_1}\ge\vec{0}\}$. Let $L = |\{i : \ell_i > -\infty\}|$ and $U = |\{i : u_i < \infty\}|$. 
\paragraph{PDHG} Primal-Dual Hybrid Gradient~\citep{chambolle2011first} iteratively takes gradient steps in primal and dual space \begin{equation*}
    \vec{x}^{(t+1)} = \proj_X\!\left(\vec{x}^{(t)} - \frac{\eta}{\omega}\big(\vec{c} - \vec{K}^\top\vec{y}^{(t)}\big)\right);~
    \vec{y}^{(t+1)} = \proj_Y\!\big(\vec{y}^{(t)} + \eta\omega\big(\vec{q}-\vec{K}\big(2\vec{x}^{(t+1)}-\vec{x}^{(t)}\big)\big)\big),
\label{eq:pdhg}\end{equation*} until a termination condition is met, where $\eta>0$ is the {\em step size} and $\omega>0$ is the {\em primal weight}. (PDHG converges whenever $\eta\le 1/\norm{\vec{K}}_2$~\citep{chambolle2016ergodic}.) 
Here, $\proj_X(\cdot)$ clamps its input to the range $[\vec{\ell},\vec{u}]\subseteq\R^n$ and $\proj_Y(\cdot)$ clamps the first $m_1$ coordinates of its input to the nonnegative orthant.

\paragraph{PDLP}
Our high-level presentation of PDLP follows the specifications of the cuPDLP algorithm due to~\citet{lu2025cupdlp}, which is a GPU-friendly version of PDLP.\footnote{The only technical differences between PDLP and cuPDLP are in the restart and termination criteria; the original version of PDLP~\citep{applegate2021practical} used a sequential trust-region method ill-suited for GPU acceleration.} The full description of (cu)PDLP, based on~\citet{lu2025cupdlp}, is given in Algorithm~\ref{alg:pdlp}. The first step is preconditioning of the LP data (Algorithm~\ref{alg:preconditioning}), which applies ten iterations of Ruiz scaling~\citep{ruiz2001scaling} followed by Pock-Chambolle scaling~\citep{pock2011diagonal}. The iterates of PDLP follow a trajectory dictated by step-size-weighted averages of PDHG steps, and the averaging is restarted whenever a trajectory is deemed to be making progress too slowly (measured via KKT error). (The moving average is also adaptive in the sense that if the actual PDHG iterate has a better KKT error, it is used instead of the averaged iterate.) Whenever a restart is triggered, the primal weight used for PDHG steps is updated via an exponential smoothing term (line~\ref{line:primal-weight-update}). Precise specifications of the termination ($\terminate$) and restart ($\mathtt{res}$) criteria, including definitions of KKT error $\kkt_{\omega}(\cdot)$, are in Appendix~\ref{app:pdlp}.
Line~\ref{line:pdhg_step} is the actual PDHG update. The adaptive PDHG procedure (detailed in Algorithm~\ref{alg:adaptivepdhg} in Appendix~\ref{app:pdlp}) conducts a line search for a step size $\eta$ that lies below a bound prescribed by the convergence theory for PDHG~\citep{chambolle2016ergodic}.

We focus on the learnability of two parameters: the {\em Pock-Chambolle} preconditioning parameter $\alpha$ (Alg.~\ref{alg:preconditioning}, line~\ref{line:pockchambolle}) and the {\em primal-weight update smoothing} parameter $\theta$ (Alg.~\ref{alg:pdlp}, line~\ref{line:primal-weight-update}).~\citet{lu2025overview} state that these are critical hyperparameters affecting performance. A key challenge, evidently, is that effects of these parameters propagate through the algorithm's entire execution trace.

\begin{algorithm}[t]
{\small
\caption{$\mathtt{PDLP}_{\alpha,\theta}(\vec{c},\vec{K},\vec{q},\vec{\ell},\vec{u})$}\label{alg:pdlp}
\begin{algorithmic}[1]
\State $\mathtt{precondition}_{\alpha}(\vec{c},\vec{K},\vec{q},\vec{\ell},\vec{u})$
\State $\vec{x}^{0,0},\vec{y}^{0,0}\leftarrow\vec{0}_n,\vec{0}_{m_1+m_2}$;~$k\leftarrow 0$;~$\widehat{\eta}^{0,0}\leftarrow 1/\norm{K}_{\infty}$;~$\omega^{0}\leftarrow\norm{\vec{c}}_2/\norm{\vec{q}}_2$ {\bf if} $\norm{\vec{c}}_2,\norm{\vec{q}}_2\ge10^{-12}$ {\bf else } $1$
\For{$s = 0,\ldots, S-1$}
    \For{$t = 0,\ldots, T-1$}
        \State $\vec{x}^{s,t+1},\vec{y}^{s,t+1},\eta^{s,t+1},\widehat{\eta}^{s,t+1}\leftarrow\mathtt{adaptivePDHG}(\vec{x}^{s,t},\vec{y}^{s,t},\omega^s,\widehat{\eta}^{s,t},k)$\label{line:pdhg_step}
        \State $\overline{\vec{x}}^{s,t+1},\overline{\vec{y}}^{s,t+1}\leftarrow\sum_{i=1}^{t+1}\eta^{s,i}\cdot(\vec{x}^{s,i},\vec{y}^{s,i})/\sum_{i = 1}^{t+1}\eta^{s,i}$
        \State $\widehat{\vec{x}}^{s,t+1},\widehat{\vec{y}}^{s, t+1}\leftarrow \vec{x}^{s, t+1}, \vec{y}^{s, t+1}$ {\bf if}~$\tfrac{\kkt_{\omega^s}(\vec{x}^{s, t+1},\vec{y}^{s, t+1})}{\kkt_{\omega^s}(\overline{\vec{x}}^{s, t+1},\overline{\vec{y}}^{s, t+1})} < 1$ {\bf else}~$\overline{\vec{x}}^{s, t+1}, \overline{\vec{y}}^{s, t+1}$

        \State $k\leftarrow k+1$
        \If{$\mathtt{term}(\vec{r}_1^{-1}*\vec{x}^{s,t+1},\vec{r}_2^{-1}*\vec{y}^{s,t+1})$~{\bf or}~$\mathtt{res}(\widehat{\vec{x}}^{s,t+1},\widehat{\vec{y}}^{s,t+1},\widehat{\vec{x}}^{s,t},\widehat{\vec{y}}^{s,t}, \vec{x}^{s,0},\vec{y}^{s,0},\omega^s)$}\label{line:termination+restart}\looseness-1
            \State {\bf break}{\color{gray}\Comment{$\vec{r}_1$ and $\vec{r}_2$ accumulate the effects of scaling in preconditioning (Alg.~\ref{alg:preconditioning})}}
        \EndIf
    \EndFor
    \State $\vec{x}^{s+1,0},\vec{y}^{s+1,0}\leftarrow \widehat{\vec{x}}^{s, t}, \widehat{\vec{y}}^{s, t}$;~$\widehat{\eta}^{s+1, 0}\leftarrow\widehat{\eta}^{s, t}$
    
    \State {\footnotesize$\omega^{s+1}\leftarrow\exp\!\big(\theta\log\tfrac{\norm{\vec{y}^{s+1,0} - \vec{y}^{s,0}}_2}{\norm{\vec{x}^{s+1,0} - \vec{x}^{s,0}}_2}+(1-\theta)\log\omega^{s}\big)$~{\bf if}~${\norm{\vec{y}^{s+1,0} - \vec{y}^{s,0}}_2,\norm{\vec{x}^{s+1,0} - \vec{x}^{s,0}}_2>10^{-12}}$~{\bf else}~$\omega^{s}$}\label{line:primal-weight-update}
\EndFor
\State \Return $(\vec{r}_1^{-1}*\vec{x}^{s,0},\vec{r}_2^{-1}*\vec{y}^{s,0})$ 
\end{algorithmic}
}
\end{algorithm}
\begin{algorithm}[t]
{\small
\caption{$\mathtt{precondition}_{\alpha}(\vec{c},\vec{K},\vec{q},\vec{\ell},\vec{u})$}\label{alg:preconditioning}
\begin{algorithmic}[1]
\State $\vec{r}_1,\vec{r}_2\leftarrow\vec{1}_{m_1+m_2},\vec{1}_{n}$
\For{$\_ = 1,\ldots,10$}
\State $\vec{D}_1\leftarrow \operatorname{diag}(\norm{\vec{K}_{1,\cdot}}_{\infty}^{-1/2},\ldots,\norm{\vec{K}_{m_1+m_2,\cdot}}_{\infty}^{-1/2})$;~$\vec{D}_2\leftarrow\operatorname{diag}(\norm{\vec{K}_{\cdot,1}}_{\infty}^{-1/2},\ldots,\norm{\vec{K}_{\cdot,n}}_{\infty}^{-1/2})$
\State  $\vec{K}\leftarrow\vec{D}_1\vec{K}\vec{D}_2$;~$\vec{c}, \vec{q}^\top, \vec{u}, \vec{\ell}\leftarrow\vec{D}_2\vec{c}, \vec{q}^\top\vec{D}_1, \vec{D}_2^{-1}\vec{u}, \vec{D}_2^{-1}\vec{\ell}$;~$\vec{r}_1,\vec{r}_2\leftarrow\vec{D}_1^{-1}\vec{r}_1, \vec{D}_2^{-1}\vec{r}_2$
\EndFor
\State $\vec{D}_1\leftarrow \operatorname{diag}(\norm{\vec{K}_{1,\cdot}}_{2-\alpha}^{-1/2},\ldots,\norm{\vec{K}_{m_1+m_2,\cdot}}_{2-\alpha}^{-1/2})$;~$\vec{D}_2\leftarrow\operatorname{diag}(\norm{\vec{K}_{\cdot,1}}_{\alpha}^{-1/2},\ldots,\norm{\vec{K}_{\cdot,n}}_{\alpha}^{-1/2})$\label{line:pockchambolle}
\State $\vec{K}\leftarrow\vec{D}_1\vec{K}\vec{D}_2$;~$\vec{c}, \vec{q}^\top, \vec{u}, \vec{\ell}\leftarrow\vec{D}_2\vec{c}, \vec{q}^\top\vec{D}_1, \vec{D}_2^{-1}\vec{u}, \vec{D}_2^{-1}\vec{\ell}$;~$\vec{r}_1,\vec{r}_2\leftarrow\vec{D}_1^{-1}\vec{r}_1, \vec{D}_2^{-1}\vec{r}_2$
\end{algorithmic}
}
\end{algorithm}

\paragraph{Generalization}

Each instance LP is given by its primal data $x = (\vec{c}, \vec{G}, \vec{h}, \vec{A}, \vec{b}, \vec{\ell}, \vec{u})$ and is drawn from an underlying distribution $D$ over LP instance space $\cX$. A {\em generalization} (or {\em sample complexity}) guarantee bounds the error $\varepsilon$ (uniformly over all parameter settings) between performance on a training set of size $N$ and expected future performance on an unseen test instance: $\mathbb{P}_{x_1,\ldots, x_N\sim D}(\sup_{\phi\in\Phi}|\frac{1}{N}\sum_{i=1}^N u_{\phi}(x_i) - \E_{x\sim D}[u_{\phi}(x)]|\le \varepsilon)\ge 1-\delta$. Generalization is controlled by {\em pseudo-dimension}~\citep{pollard2012convergence}: if $\cU = \{u_{\phi} : \cX\to [0, T] \mid \phi\in\Phi\}$ is the class of (bounded) performance measures (in our case, iteration count) with hyperparameter space $\Phi$, $\Pdim(\cU)$ is the largest integer $N$ for which there exist $N$ LP instances $x_1,\ldots, x_N$ and $N$ thresholds $\gamma_1,\ldots,\gamma_N$ such that for every sign pattern $\sigma\in\{-1,1\}^N$, there exists $\phi\in\Phi$ such that $\mathrm{sign}(u_{\phi}(x_i) - \gamma_i) = \sigma_i$ for all $i$. For a fixed training set size $N$, the resulting generalization error is $\varepsilon_{\cU}(N,\delta) \le T\sqrt{(\Pdim(\cU) +\log(1/\delta))/N}$. For a fixed error threshold $\varepsilon$, the {\em sample complexity} is $N_{\cU}(\varepsilon,\delta) = O((T/\varepsilon)^2(\Pdim(\cU) + \log(1/\delta)))$.

\section{Tuning the Step Size and Primal Weight in Vanilla PDHG}\label{sec:pdhg}

In this section we study the sample complexity of directly tuning the step size $\eta$ and primal weight $\omega$ of the base PDHG algorithm. The performance functions, which measure the number of PDHG iterations, are given by $u_{\eta,\omega}(\vec{c},\vec{G},\vec{h},\vec{A},\vec{b},\vec{\ell},\vec{u}) = \min\{t \le T : \terminate(\vec{x}^{(t)}, \vec{y}^{(t)})=\mathtt{true}\}$ if such a $t$ exists; $T$ otherwise, where $\vec{x}^{(t)}, \vec{y}^{(t)}$ are the PDHG iterates.

Lemma~\ref{lemma:vanilla_pdhg_iterates} pins down the structure of the iterates $\vec{x}^{(t)}, \vec{y}^{(t)}$ at each round as $(\eta,\omega)$ vary. The analysis is similar to recent generalization guarantees for parameter tuning in gradient descent~\citep{sharma2025gradient}, but we must reckon with the additional structural complications due to projections.

\begin{lemma}\label{lemma:vanilla_pdhg_iterates} 
    There are $O(2^{2T}(L+U+m_1)^{2T}T^{4T})$ algebraic curves of degree $\le 2T$ that partition $\R_{\ge 0}^2$ into connected components such that for each connected component $\cC$, there exist rational functions $g_{i,t}^{\cC}, h_{j,t}^{\cC} : \cC\to\R$, for all $i \in \{1,\ldots, n\}$, $j \in \{1,\ldots, m_1+m_2\}$, and $t \in \{0,\ldots, T\}$,  with $\deg(g_{i,t}^{\cC}),\deg(h_{j,t}^{\cC})\le 2t$, such that $x_i^{(t)} = g_{i,t}^{\cC}(\eta,\omega)$ and $y_j^{(t)} = h_{j,t}^{\cC}(\eta,\omega)$ for all $(\eta,\omega)\in\cC$.
\end{lemma}

\begin{proof} The proof is by induction on $t$.
Initial primal-dual iterates are constant (degree-zero) in $\eta,\omega$. Now, fix a particular connected component $\cC$ from the partition of $\R^2_{\ge 0}$ obtained at round $t$.
Expanding the primal update prior to projection yields, for all $i = 1,\ldots,n$, $\widetilde{x}_i^{(t+1)} \coloneq x_i^{(t)} - \frac{\eta}{\omega}\left(c_i - \vec{K}_{\cdot,i}^\top\vec{y}^{(t)}\right) = x_i^{(t)} - c_i\frac{\eta}{\omega} + \textstyle\sum_{j=1}^{m_1+m_2}K_{j,i}y_j^{(t)}\frac{\eta}{\omega}.$
If $\widetilde{x}_i^{(t+1)}\le\ell_i$, $x_i^{(t+1)} = \ell_i$ and if $\widetilde{x}_i^{(t+1)}\ge u_i$, $x_i^{(t+1)} = u_i$. Otherwise, $x_i^{(t+1)} = \widetilde{x}_i^{(t+1)}$. So, the two algebraic curves $\widetilde{x}_i^{(t+1)}-\ell_i = 0$ and $\widetilde{x}_i^{(t+1)}-u_i = 0$---which are both of degree at most $2+2t = 2(t+1)$ by the inductive hypothesis---further partition $\cC$ such that in each connected component of $\cC$, $x_i^{(t+1)}$ is a rational function of $\eta,\omega$ of degree at most $2(t+1)$. (If either $\ell_i = -\infty$ or $u_i = \infty$ we do not include those curves in our partition of $\R_{\ge 0}^2$.)
We now expand the dual update---prior to projection---under the assumption that the projection for each $x_i^{(t)}$ is inactive. This assumption is without loss of generality for our degree analysis, since if a projection is active for any $x_i^{(t)}$ the degree can only decrease. We have, for all $j = 1,\ldots, m_1+m_2$,
$\widetilde{y}^{(t+1)}_j \coloneq y_j^{(t)} + \eta\omega(q_j - \vec{K}_{j,\cdot}^\top(2\vec{x}^{(t+1)} - \vec{x}^{(t)})) 
        = y_j^{(t)} + \eta\omega(q_j - \textstyle\sum_{i=1}^n K_{j, i}\big(2x_i^{(t+1)} - x_i^{(t)}\big)) 
        = y_j^{(t)} + \eta\omega(q_j - \textstyle\sum_{i=1}^{n}K_{j, i}(2(x_i^{(t)} - \frac{\eta}{\omega}(c_i - \vec{K}_{\cdot,i}^\top\vec{y}^{(t)})) - x_i^{(t)}))
        =y_j^{(t)} + \eta\omega q_j - \textstyle\sum_{i=1}^{n}(2K_{j, i}\eta\omega x_i^{(t)} - 2K_{j, i}\eta^2c_i + \textstyle\sum_{j'=1}^{m_1+m_2}2K_{j, i}K_{j',i}\eta^2y_{j'}^{(t)}- K_{j, i}\eta\omega x_i^{(t)}).$
For $j = 1,\ldots, m_1$, if $\widetilde{y}_j^{(t+1)}\le 0$ then $y_j^{(t+1)} = 0$, and $y_j^{(t+1)} = \widetilde{y}_j^{(t+1)}$ otherwise. For $j = m_1+1,\ldots, m_1+m_2$, $y_j^{(t+1)} = \widetilde{y}_j^{(t+1)}$. So, the algebraic curve $\widetilde{y}_i^{(t+1)} = 0$, which is of degree at most $2+2t = 2(t+1)$ by the inductive hypothesis, further partitions $\cC$ such that in each connected component of $\cC$, $y_j^{(t+1)}$ is a rational function of $\eta,\omega$ of degree at most $2(t+1)$. A careful counting argument (see Appendix~\ref{app:pdhg}) completes the proof. 
\end{proof}

Lemma~\ref{lemma:vanilla_pdhg_termination} analyzes the termination criteria, which requires primal feasibility, dual feasibility, and sufficiently small duality gap of the candidate iterates. (Precise definitions are in Appendix~\ref{app:pdlp}, but the key metrics are defined within the proceeding proof.) A key aspect of the analysis is considering decision boundaries so that the dual variable $\vec{\lambda} = \proj_{\Lambda_1\times\cdots\Lambda_n}(\vec{c} - \vec{K}^\top\vec{y})$ is well-behaved.

\begin{lemma}\label{lemma:vanilla_pdhg_termination}
    Fix a connected component $\cC$ in the partition of $\R_{\ge 0}^2$ established in Lemma~\ref{lemma:vanilla_pdhg_iterates}. There are $n^{O(1)} (2(L+U+m_1)T)^{O(T)}$ additional curves that partition $\cC$ such that for all $(\eta,\omega)$ within a connected component of $\cC$, $\terminate(\vec{x}^{(t)}, \vec{y}^{(t)})$ is invariant for all $t\le T$.
\end{lemma}

\begin{proof}
    Fix a connected component $\cC$ given by Lemma~\ref{lemma:vanilla_pdhg_iterates}. In this cell, $x_i^{(t)}$ and $y_i^{(t)}$ are rational functions of $\eta,\omega$ of degree at most $2t$, for all $t = 0,\ldots, T$. We aim to find a further refinement of this component such that for all $\eta,\omega$ in a connected component of the refinement, primal feasibility, dual feasibility, and duality gap of all iterates $(\vec{x}^{(t)},\vec{y}^{(t)})$ are invariant.
    
    First, we handle primal feasibility, which is the condition $\norm{\vec{A}\vec{x}^{(t)} - \vec{b}}_2^2 + \norm{(\vec{h} - \vec{G}\vec{x}^{(t)})^+}_2^2 \le \gamma_p^2$, where $\gamma_p\coloneq \varepsilon_{\tol}(1+\norm{\vec{q}}_2)$. 
    For each $j = 1,\ldots, m_1$, add the curve $h_j - \vec{G}_{j, \cdot}^\top\vec{x}^{(t)} = 0$ to $\cC$. (Since each $x_i^{(t)}$ is a rational function of degree $\le 2t$, $h_j - \vec{G}_{j,\cdot}^{\top}\vec{x}^{(t)} = 0$ is an algebraic curve---after clearing denominators---of degree at most $2t$.) In each of the connected components induced by these curves, the sign of $h_j - \vec{G}_{j,\cdot}^\top\vec{x}^{(t)}$ is invariant and hence $(h_j - \vec{G}_{j,\cdot}^\top\vec{x}^{(t)})^+$ is a degree $2t$ rational function.
    Finally, within each such component, add the curve $\norm{\vec{A}\vec{x}^{(t)} - \vec{b}}_2^2 + \sum_{j=1}^{m_1}((h_j-\vec{G}_{j, \cdot}^\top\vec{x}^{(t)})^+)^2 \le \gamma_p^2,$ which is of degree $4t$. Now, within each component, primal feasibility of $(\vec{x}^{(t)},\vec{y}^{(t)})$ is invariant.

    To handle dual feasibility and duality gap, we need to get a handle on dual variable $\vec{\lambda}^{(t)} = \proj_{\Lambda_1\times\cdots\times\Lambda_n}(\vec{c}-\vec{K}^\top\vec{y}^{(t)})$ first. For that, we put down, for each $i\in\{1,\ldots,n\}$ such that either $\ell_i = -\infty$ and $u_i\in\R$ or $\ell_i\in\R$ and $u_i=\infty$, the degree $2t$ curve $c_i - \vec{K}_{\cdot, i}^\top\vec{y}^{(t)} = 0$. Within each connected component of $\cC$ induced by these curves, $\lambda^{(t)}_i$ is thus a degree $2t$ rational function in $(\eta,\omega)$.
    Now, with $\gamma_d\coloneq \varepsilon_{\tol}(1+\norm{\vec{c}}_2)$, dual feasibility is the condition $\sum_{i=1}^n (c_i - \vec{K}_{\cdot, i}^\top\vec{y}^{(t)} - \lambda_i^{(t)})^2\le\gamma_d^2,$
    which is a degree $4t$ curve (after clearing denominators). Within each region induced by these curves, dual feasibility of $(\vec{x}^{(t)},\vec{y}^{(t)})$ is invariant.
    Finally, we handle the duality gap condition, which reads $|\vec{q}^\top\vec{y}+\vec{\ell}^\top\vec{\lambda}^{+}-\vec{u}^\top\vec{\lambda}^{-}-\vec{c}^\top\vec{x}| \le\gamma_g$, where $\gamma_g \coloneq \varepsilon_{\tol}(1+\abs{\vec{q}^\top\vec{y}+\vec{\ell}^\top\vec{\lambda}^+-\vec{u}^\top\vec{\lambda}^-}+\abs{\vec{c}^\top\vec{x}})$. Put down curves $\lambda_i^{(t)} = 0$ for all $i = 1,\ldots, n$. In cells where $\lambda_i^{(t)}< 0$ add the curves $\vec{q}^\top\vec{y}^{(t)} - \vec{u}^\top\vec{\lambda}^{(t)} - \vec{c}^\top\vec{x}^{(t)} = 0$ and $\vec{q}^\top\vec{y}^{(t)} - \vec{u}^\top\vec{\lambda}^{(t)} = 0$ and in cells where $\lambda_i^{(t)}\ge 0$ add the curves $\vec{q}^\top\vec{y}^{(t)} + \vec{\ell}^\top\vec{\lambda}^{(t)} - \vec{c}^\top\vec{x}^{(t)} = 0$ and $\vec{q}^\top\vec{y}^{(t)} + \vec{\ell}^\top\vec{\lambda}^{(t)} = 0$. Add the additional curve $\vec{c}^\top\vec{x}^{(t)} = 0$. Within the induced components, the signs of all absolute value terms in the duality gap condition are invariant. So, within each component the duality gap condition is an algebraic curve of degree $O(t)$, which induces the final partition over which the truth of the duality gap condition is invariant. Finally, collecting all curves over all rounds $t \le T$ (and carefully counting them), we obtain a partition such that $\terminate(\vec{x}^{(t)}, \vec{y}^{(t)})$ is invariant in each connected component for all $t$. 
\end{proof}

We can finally apply the general result of~\citet{balcan2024much} to obtain a pseudo-dimension bound.

\begin{theorem}\label{theorem:pdhg}
    $\Pdim(\{u_{\eta,\omega} : \eta,\omega > 0\}) = O(T\log(m_1nT))$.
\end{theorem}

\section{Tuning PDLP Parameters}\label{sec:pdlp}

In this section we analyze the sample complexity of tuning key hyperparameters of (cu)PDLP (Alg.~\ref{alg:pdlp}). Our analysis builds on recent advances in data-driven algorithm design for function classes with \emph{Pfaffian structure}, which strictly generalizes polynomial structure. Understanding the Pfaffian structure of the numerical computations in Alg.~\ref{alg:pdlp} is what allows us to characterize the behavior of the full execution of PDLP.
We summarize here the key tools from prior work needed for our analysis. \looseness-1

\begin{definition}[Pfaffian chains and functions~\citep{khovanskiui1991fewnomials}]
Let $U \subseteq \mathbb{R}^d$ be open. A sequence of functions
$f_1, f_2, \dots, f_q : U \to \mathbb{R}$
forms a \emph{Pfaffian chain} of length $q$ and degree at most $\alpha$ if for every $i \in \{1,\ldots,q\}$ and $j \in \{1,\ldots,d\}$,
$\frac{\partial f_i}{\partial x_j}(x)
=
P_{i,j}(x, f_1(x), \dots, f_i(x)),$
where each $P_{i,j}$ is a polynomial of degree at most $\alpha$.
A function $f : U \to \mathbb{R}$ is \emph{Pfaffian} with respect to this chain if
$
f(x) = Q(x, f_1(x), \dots, f_q(x)),
$
for some polynomial $Q$ of degree at most $\beta$.
We say $f$ has {\em Pfaffian complexity} $(q, \alpha, \beta)$.
\end{definition}

Any polynomial is Pfaffian with $q=0$. Exponentials $e^{ax}$ and logarithms $\log(ax)$ are Pfaffian over their domains. Periodic trigonometric functions like $\sin(ax)$ and $\cos(ax)$ are {\em not} Pfaffian. 
Pfaffian functions are closed under addition, multiplication, division (away from zero), and composition with $\exp$ and $\log$~\citep{balcan2025algorithm}. These properties will be used repeatedly.
%

\begin{definition}[Pfaffian piecewise structure~\citep{balcan2025algorithm}]
Function $f : \Phi \subseteq \mathbb{R}^d \to \mathbb{R}$ is \emph{$(k_F, k_G, q, M, \Delta, d)$-Pfaffian piecewise structured} if there exist $k_G$ Pfaffian {\em boundary} functions
and $k_F$ Pfaffian {\em piece} functions---all with Pfaffian complexity $(q, M, \Delta)$---such that on each region of $\Phi$ induced by the $k_G$-bit sign pattern of the boundary functions, $f$ coincides with a piece function.
\end{definition}

A key learning-theoretic result we will need is the following bound on the pseudo-dimension of piecewise-Pfaffian functions in terms of their structural complexity.

\begin{theorem}[Pseudo-dimension of Pfaffian piecewise classes~\citep{balcan2025algorithm}]
\label{thm:pfaffian-pdim}
Let $\mathcal{U} = \{u_\phi : \mathcal{X} \to \mathbb{Z}\}_{\phi \in \Phi \subseteq \mathbb{R}^d}$.
If for every $x \in \mathcal{X}$, the dual function $u_x^*(\phi) \coloneq u_\phi(x)$ is $(k_F, k_G, q, M, \Delta, d)$-Pfaffian piecewise structured, then
$\mathrm{Pdim}(\mathcal{U})\le
d^2q^2
+ 2dq \log(\Delta + M)
+ 2 \log(\Delta (k_F + k_G))
+ 16.$
\end{theorem}

\subsection{Main Pseudo-dimension Bounds: Tuning $\theta$ and $\alpha$ Individually and Simultaneously}

First, we analyze the dependence of PDLP on the smoothing parameter $\theta$ (Algorithm~\ref{alg:pdlp}; line~\ref{line:primal-weight-update}). Full proofs of results are in Appendix~\ref{app:pdlp}. We assume LP instances are such that $\norm{\vec{K}}_{\infty} > 1$.

\begin{theorem}[Pseudo-dimension bound for tuning $\theta$ in PDLP]
\label{thm:pdim-theta}
Fix an LP instance $x=(\vec{c},\vec{K},\vec{q},\vec{\ell},\vec{u})\in\cX$. Let $u_\theta:\mathcal X\to\mathbb{Z}$ be the number of iterations of Algorithm~\ref{alg:pdlp} using primal-weight update smoothing parameter $\theta$ (with all other hyperparameters---in particular the Pock-Chambolle $\alpha$---fixed to their default settings). 
Then, for each LP instance $x\in \mathcal X$, the dual function $u^*_{x}(\theta) \coloneq u_\theta(x)$
is $(k_F, k_G, q, M, \Delta, 1)$-Pfaffian piecewise structured with:
\[\begin{aligned}
&k_F \le 2^{k_G},\\
&k_G \le (ST)^{2.3}\log\norm{\vec{K}}_2(L + U + m_1 + 6 ) + 6S, \\ 
&q = O((ST)^{2.3}\log\norm{\vec{K}}_2),\\ 
&M, \Delta \le 2^{O((ST)^{2.3}\log\norm{\vec{K}}_2)}.
\end{aligned}\] Consequently, 
$$
\operatorname{Pdim}(\{u_\theta:\mathcal X\to\mathbb{Z}\mid \theta\in\Theta\})
=
O((ST)^{4.6}\widetilde{K}^2
+
(ST)^{2.3}\widetilde{K}(L+U+m_1)),
$$ where $\widetilde{K}=\sup_{\vec{K}\in\cX}\log\norm{\vec{K}}_2$ and $S, T$ are the outer and inner loop iteration limits of Algorithm~\ref{alg:pdlp}.
\end{theorem}

\begin{proof}[Proof sketch]
All branching decisions (projections, decisions in the adaptive step-size procedure, the KKT error comparison in line~7, the termination and restart checks in line~9) can be written as sign conditions of the form $\tau(\theta)\ge 0$.
These define $k_G$ Pfaffian boundary functions.

Fix a region where all predicates have constant sign. The algorithm follows a fixed execution path within such a region. We then induct over iterations: primal and dual updates (line~5) are algebraic operations, averaging (line~6) involves summation and division, and the primal-weight update (line~12) involves $\log$ and $\exp$. All operations preserve Pfaffian structure.

Thus, all iterates and outputs are Pfaffian on each region. Counting sign patterns yields $k_F \le 2^{k_G}$. Pfaffian complexity bounds follow from carefully counting all operations. A key intermediate technical result we establish is a bound on the number of iterations in the adaptive step size procedure (Algorithm~\ref{alg:adaptivepdhg}), which performs a line search for the PDHG step size. While the original PDLP specification seems to leave open the possibility of an unbounded line search, in Lemma \ref{lemma:adaptive-pdhg-iter-bound} we show that the number of iterations is at most $O((ST)^{1.3}\widetilde{K})$. Finally, applying Theorem \ref{thm:pfaffian-pdim} to the established piecewise-Pfaffian structure gives the desired pseudo-dimension bound.
\end{proof}

Next, we present our main structural result for individual tuning of $\alpha$.

\begin{theorem}[Pseudo-dimension bound for tuning $\alpha$ in  PDLP]
\label{thm:pdim-alpha-formal}
Fix an LP instance $x = (\vec{c},\vec{K},\vec{q},\vec{\ell},\vec{u})$. Let $u_{\alpha}:\cX\to\Z$ be the number of iterations of Algorithm~\ref{alg:pdlp} using Pock-Chambolle preconditioning parameter $\alpha$ (with all other hyperparameters---in particular the primal-weight update smoothing parameter $\theta$---fixed to their default settings). Then, for each LP instance $x\in\cX$ for which every row norm and column norm used in the final Pock–Chambolle scaling step (Algorithm \ref{alg:preconditioning}) is strictly positive, the dual function $u^*_x(\alpha)\coloneq u_{\alpha}(x)$ is $(k_F,k_G,q,M,\Delta,1)$-Pfaffian piecewise structured, with
\[\begin{aligned}
    k_F\le 2^{K_\alpha},~k_G\le K_\alpha,~q = O(n+m+(ST)^{2.3}\log\norm{\vec{K}}_2),~M,\Delta \le 2^{O((ST)^{2.3}\log\norm{\vec{K}}_2)}
\end{aligned}\] where $m = m_1 + m_2$ and $K_\alpha = (ST)^{2.3}\log\norm{\vec{K}}_2(L+U+m_1+6)+6S$.
Consequently, 
$$\operatorname{Pdim}(\{u_\alpha:\mathcal X\to\mathbb{Z}\mid \alpha\in A\})
=
O((n+m+(ST)^{2.3}\widetilde{K})^2
+
(ST)^{2.3}\widetilde{K}(L+U+m_1)),$$ where $\widetilde{K}=\sup_{\vec{K}\in\cX}\log\norm{\vec{K}}_2$ and $S, T$ are the outer and inner loop iteration limits of Algorithm~\ref{alg:pdlp}.
\end{theorem}

\begin{proof}[Proof sketch]
The ten Ruiz iterations are independent of $\alpha$---let $\vec{K}$ be the post-Ruiz matrix. Pock-Chambolle scaling, with $\mathsf{r}_j \coloneq \norm{{\vec{K}}_{j,\cdot}}_2 > 0$ and $\mathsf{c}_i \coloneq\norm{{\vec{K}}_{\cdot,i}}_2 > 0$, is given by preconditioners $\vec{D}_1(\alpha)=\operatorname{diag}(\mathsf{r}_1^{2-\alpha},\dots,\mathsf{r}_m^{2-\alpha})$,
$\vec{D}_2(\alpha)=\operatorname{diag}(\mathsf{c}_1^{\alpha},\dots,\mathsf{c}_n^{\alpha})$.
These computations represent one-variable Pfaffian functions of $\alpha$, since
\[
\mathsf{r}_j^{2-\alpha}=\mathsf{r}_j^2 e^{-\alpha\log \mathsf{r}_j},
\qquad
\mathsf{c}_i^\alpha=e^{\alpha\log \mathsf{c}_i},
\qquad
\mathsf{c}_i^{-\alpha}=e^{-\alpha\log \mathsf{c}_i}.
\]
So, all entries of the preconditioned data ${\vec{K}}(\alpha)\leftarrow\vec{D}_1(\alpha)\vec{K}\vec{D}_2(\alpha)$, $\vec{c}(\alpha)\leftarrow\vec{D}_2(\alpha)\vec{c}$, $\vec{q}(\alpha)^\top\leftarrow\vec{q}^\top \vec{D}_1(\alpha)$, $\vec{u}(\alpha)\leftarrow\vec{D}_2(\alpha)^{-1}\vec{u}$, $\vec{\ell}(\alpha)\leftarrow\vec{D}_2(\alpha)^{-1}\vec{\ell}$, as well as the accumulated scaling vectors $\vec{r}_1(\alpha),\vec{r}_2(\alpha)$ used to unscale the iterates for termination checks, are Pfaffian in $\alpha$.
It follows that the entire preconditioned input to Algorithm~\ref{alg:pdlp} is Pfaffian in $\alpha$ from the
Pfaffian chain
$$
\mathcal{C}
=\left(
\alpha,
\mathsf{r}_1^{2-\alpha},\ldots,\mathsf{r}_m^{2-\alpha},
\mathsf{c}_1^{\alpha},\ldots, \mathsf{c}_n^{\alpha},
\mathsf{c}_1^{-\alpha},\ldots, \mathsf{c}_n^{-\alpha}\right),
$$ of length $ 1 + m + 2n = O(n+m)$.
The remainder of the analysis proceeds as in the $\theta$ case, preserving Pfaffian structure.
Applying Theorem~\ref{thm:pfaffian-pdim} yields the result.
\end{proof}

Finally, we present our pseudo-dimension bound for 
tuning $\theta$ and $\alpha$ simultaneously (Appendix \ref{app:extensions} contains the full formal theorem statement (Theorem \ref{thm:pdim-theta-alpha-full}) and its proof).

\begin{theorem}[Pseudo-dimension bound for jointly tuning $(\theta,\alpha)$ in PDLP]
\label{thm:pdim-theta-alpha-main}
Let $\cU_{\Theta,A} = \{u_{\theta,\alpha} : \cX\to\mathbb{Z}\mid (\theta,\alpha)\in\Theta\times A\}$ denote the class of functions measuring iteration count of Algorithm~\ref{alg:pdlp} with primal-weight update smoothing parameter $\theta$ and Pock-Chambolle preconditioning parameter $\alpha$. Then, under mild assumptions, $$\operatorname{Pdim}(\cU_{\Theta,A}) = O\big((n+m+(ST)^{2.3}\widetilde{K})^2
+
(ST)^{2.3}\widetilde{K}(L+U+m_1)\big),$$ where 
$\widetilde{K}=\sup_{\vec{K}\in\cX}\log\norm{\vec{K}}_2$ and $S, T$ are the outer and inner loop iteration limits of Algorithm~\ref{alg:pdlp}. 
\end{theorem}

The essential takeaway from Theorems~\ref{thm:pdim-theta},~\ref{thm:pdim-alpha-formal}, and~\ref{thm:pdim-theta-alpha-main} is that the primal-weight update smoothing parameter and the Pock-Chambolle preconditioning parameter are learnable with sample complexity that is polynomial in all problem parameters. That is, a polynomial-size training set of LP instances suffices to ensure that a $(\theta,\alpha)$ setting that yields good performance over the training set will also lead to good expected performance over future test instances. Critically, these guarantees apply independent of the actual procedure used to select hyperparameters $\theta,\alpha$ based on the training set.

\section{Experiments}\label{sec:experiments}

We demonstrate via experiments that tuning the primal weight update smoothing parameter $\theta$ and the Pock-Chambolle parameter $\alpha$ in a data-dependent manner can lead to improved convergence rates (in terms of both iteration count and run-time). Furthermore, parameter tuning ought to be done in a {\em data-dependent} manner---the optimal set of parameters for one instance distribution can be completely different (or even opposite) from the optimal set of parameters for a different instance distribution.\looseness-1

We summarize here the LP instance distributions we use in experiments.
(1) QAP15 is an LP in the well-known Mittelmann benchark suite~\citep{mittelmann2025latest} with 6331 rows, 22275 columns, and 110700 nonzeros. It is a linearization of a quadratic assignment problem~\citep{adams1994improved}. We draw instances by randomly perturbing the objective vector of QAP15.
(2) Transportation instances: the objective is to find minimum cost shipments between a set of suppliers and a set of customers subject to demand not exceeding supply~\citep{maaz2025pdlp}. We consider instances with 30 suppliers and 45 customers (1350 variables, 75 constraints), and 60 suppliers and 90 customers (5400 variables, 150 constraints). Suppliers' supplies and customers' demands are drawn randomly as in~\citet{maaz2025pdlp}. 
(3) Winner determination in combinatorial auctions: these are LP relaxations of weighted-set packing integer programs that compute a revenue-maximizing set of non-overlapping package bids in a multi-item auction. We consider three benchmark distributions: decay-decay~\citep{sandholm2002winner} with 100 items and 1000 bids (1000 variables, 100 constraints), and 200 items and 2000 bids (2000 variables, 200 constraints); and multipaths~\citep{leyton2000towards} with 90 items and 2000 bids (2000 variables, 90 constraints).

\begin{figure}[!tb]
    \centering
     \includegraphics[trim={0cm 0cm 0cm 0cm},clip, width=0.495\linewidth]{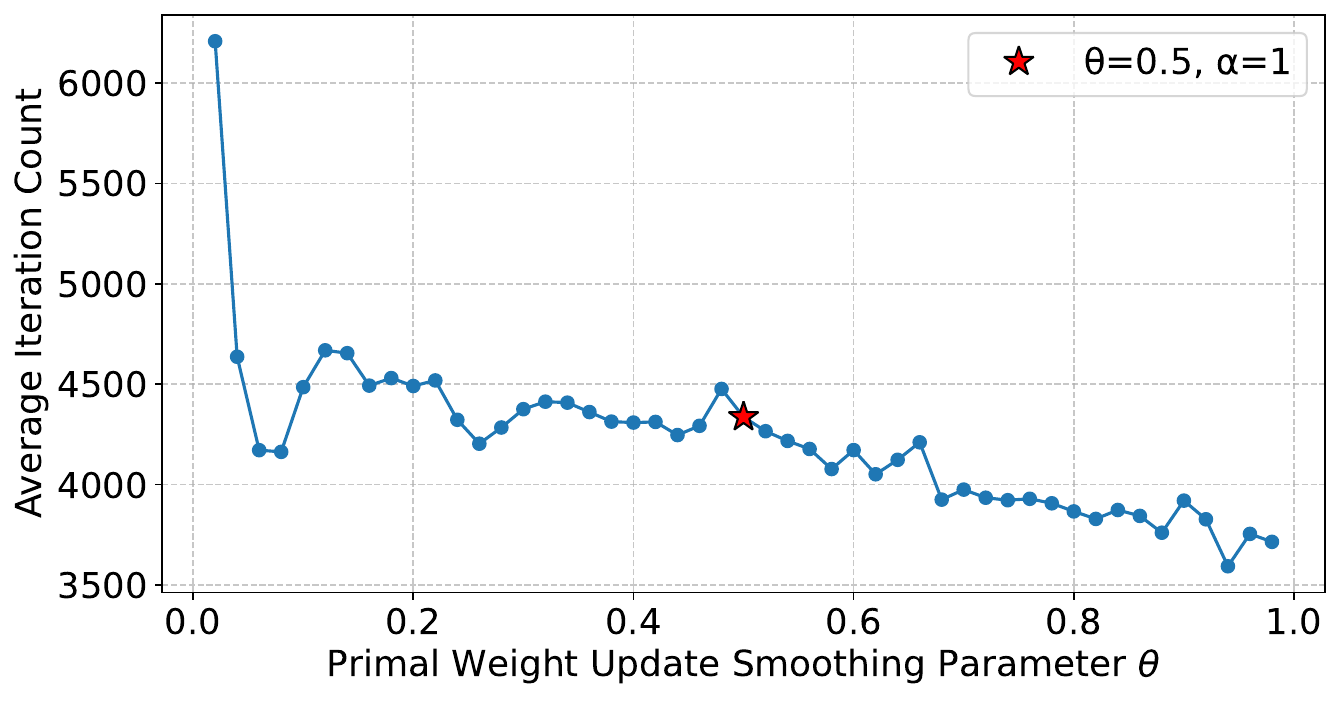}
    \includegraphics[trim={0cm 0cm 0cm 0cm},clip, width=0.495\linewidth]{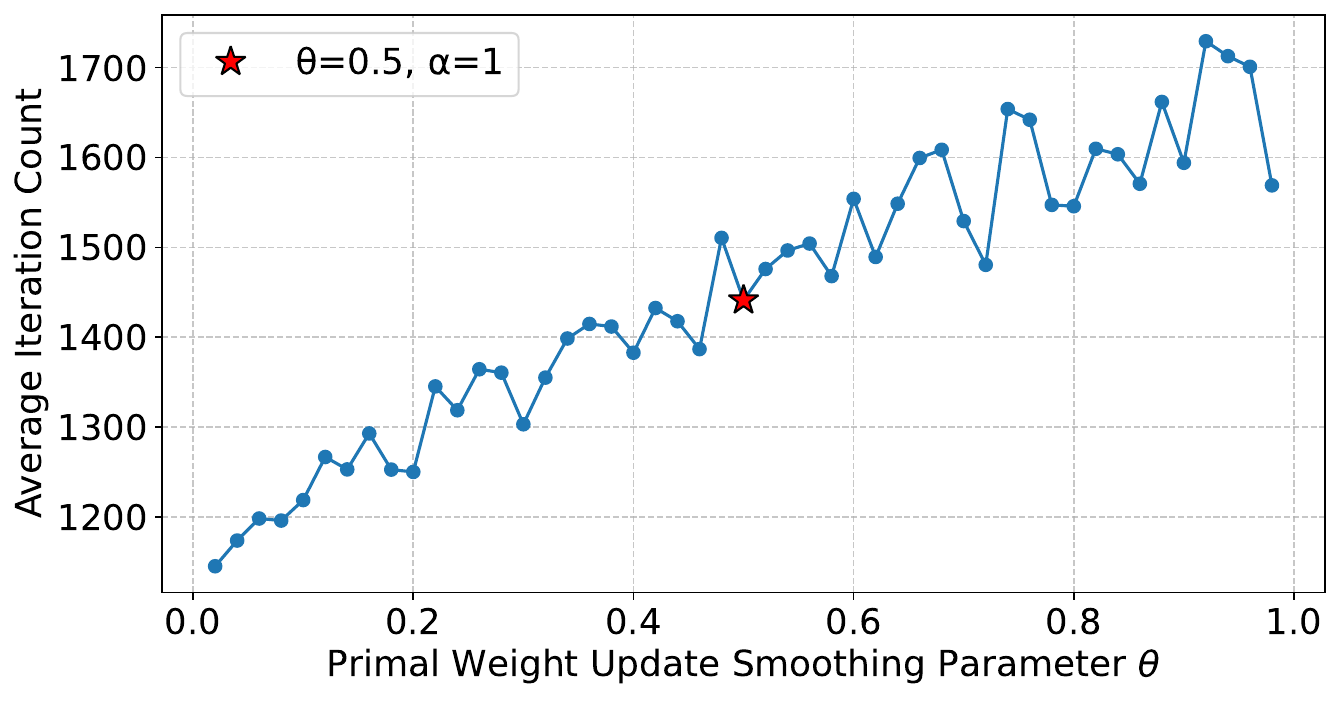}
    \includegraphics[trim={0cm 0cm 0cm 0cm},clip, width=0.495\linewidth]{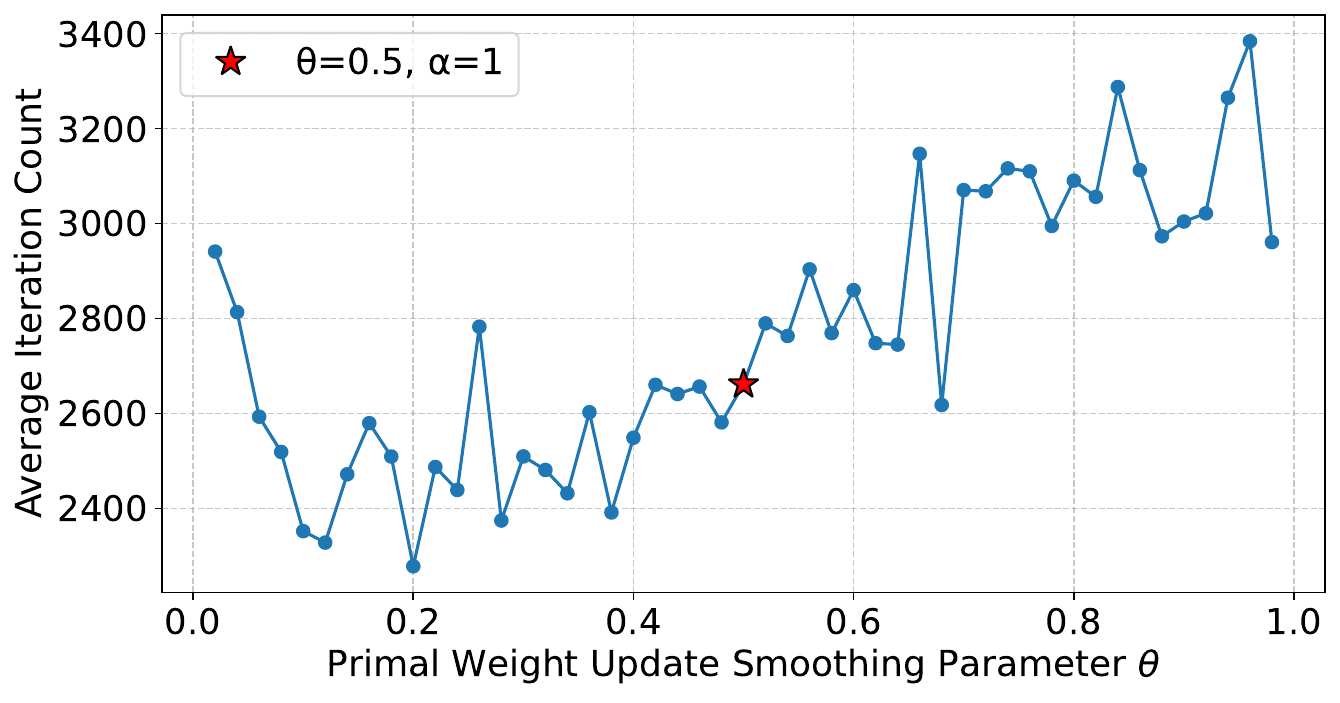}
    \includegraphics[trim={0cm 0cm 0cm 0cm},clip, width=0.495\linewidth]{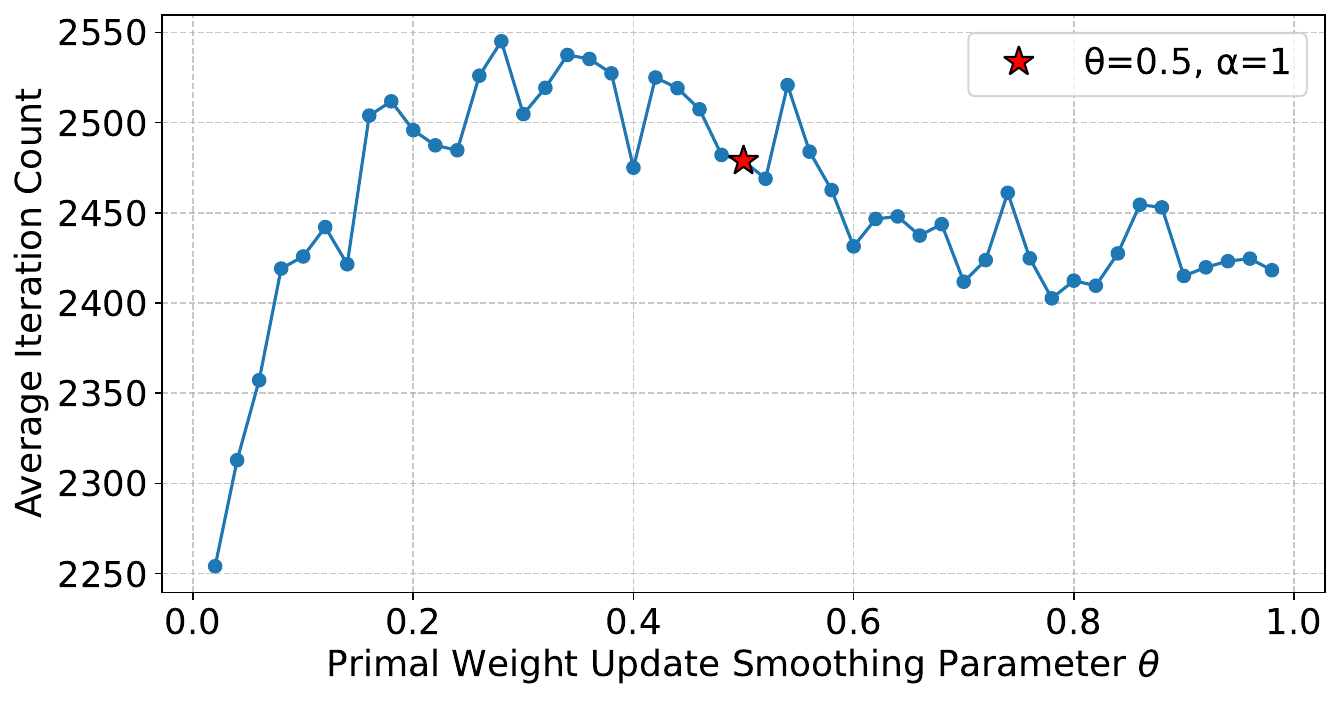}
    \caption{Tuning $\theta$ with $\alpha = 1$; iteration count (averaged over 100 instances). {\bf (Top left)} Transportation, 60 suppliers, 90 customers; {\bf (Top right)} Decay-decay auction, 200 items, 2000 bids; {\bf (Bottom left)} Multipaths auction, 90 items, 2000 bids; {\bf (Bottom right)} Perturbed QAP15.}
    \label{fig:theta_iteration_plots}
\end{figure}

\begin{figure}[!htb]
    \centering
     \includegraphics[trim={0cm 0cm 0cm 0cm},clip, width=0.495\linewidth]{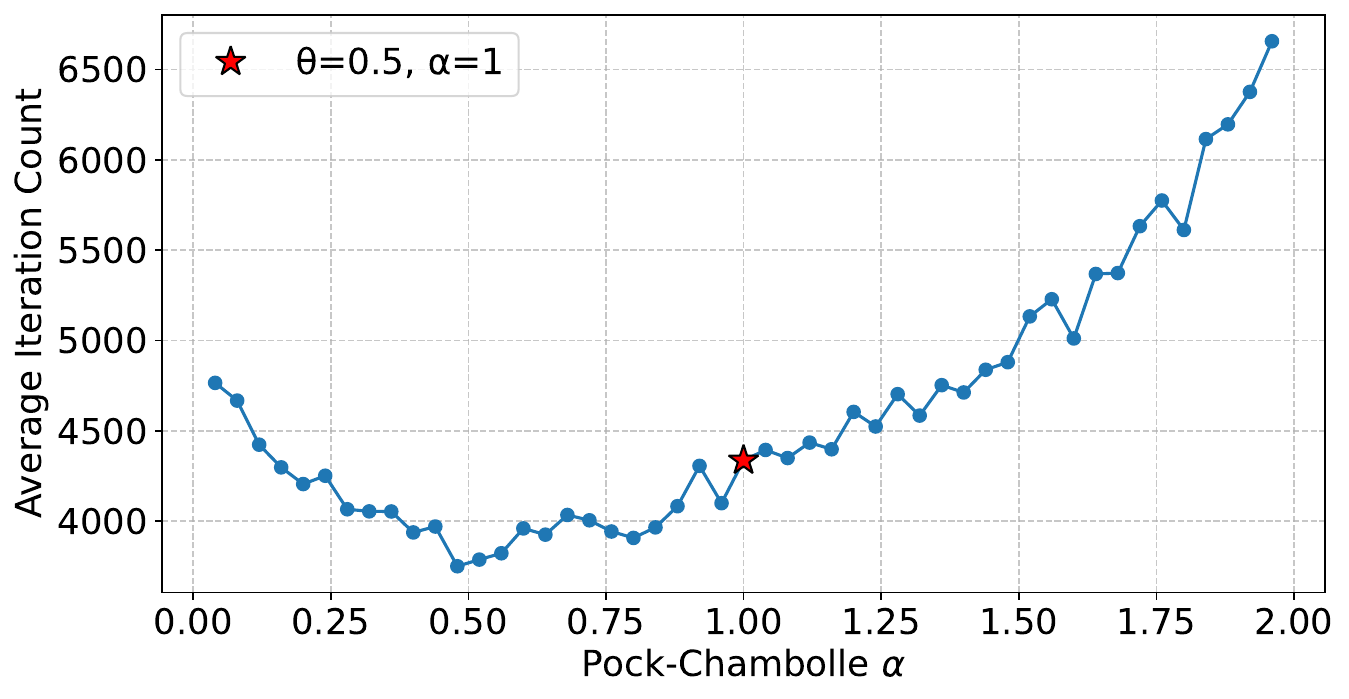}
    \includegraphics[trim={0cm 0cm 0cm 0cm},clip, width=0.495\linewidth]{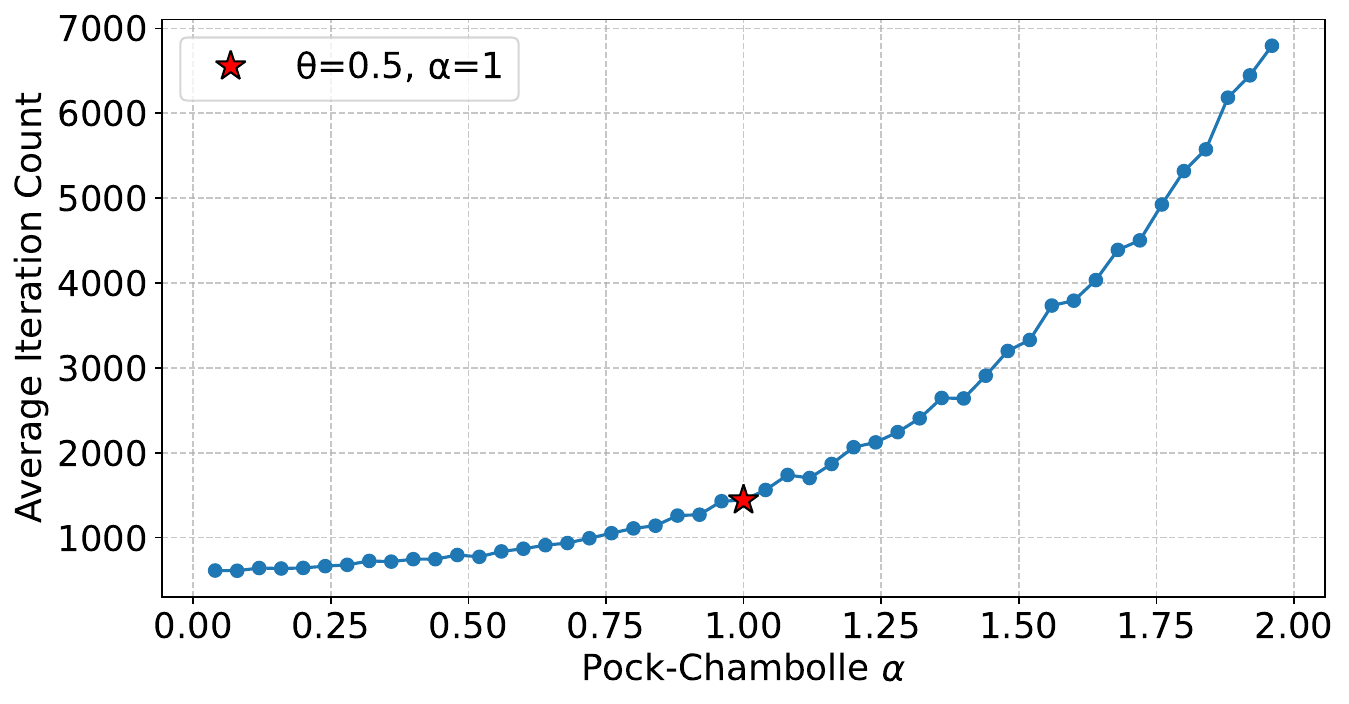}
    \caption{Tuning $\alpha$ with $\theta = 1$; iteration count (averaged over 100 instances). {\bf (Left)} Transportation, 60 suppliers, 90 customers; {\bf (Right)} Decay-decay auction, 200 items, 2000 bids.}
    \label{fig:alpha_iteration_plots}
    \vspace{-1.1em}
\end{figure}

In experiments we used an open source PyTorch implementation of PDLP~\citep{maaz2025pdlp} that exactly mirrors the specifications in Algorithms~\ref{alg:pdlp},~\ref{alg:preconditioning}, and~\ref{alg:adaptivepdhg} (other implementations~\citep{applegate2021practical,lu2025cupdlp} employ additional standard techniques such as presolve that we do not analyze in this paper). All experiments were run in Google Colab on a NVIDIA G4 (RTX PRO 6000 Blackwell Server Edition) GPU. We set $\varepsilon_{\tol} = 10^{-4}$ (the error tolerance in the termination criteria; see Appendix~\ref{app:pdlp}). PDLP default settings are $\theta = 0.5$ and $\alpha = 1$, marked by a red star in plots.

\begin{figure}
    \centering
    \includegraphics[trim={0 0 0 0},clip,width=0.495\linewidth]{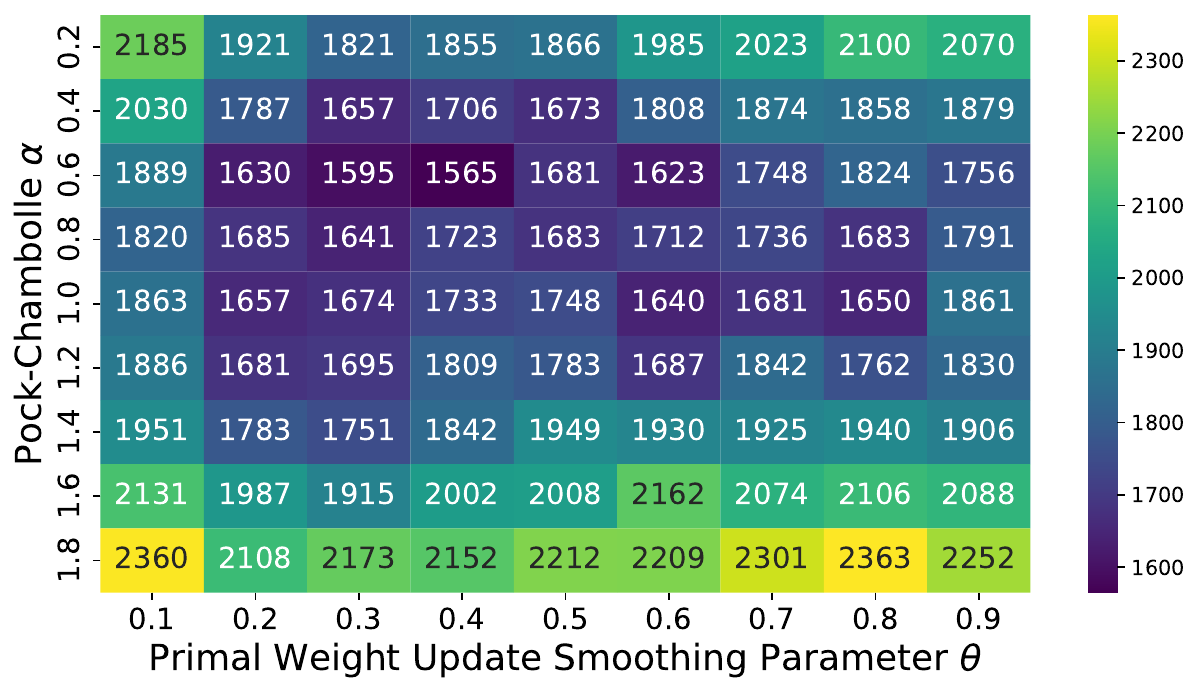}
    \includegraphics[trim={0 0 0 0},clip,width=0.495\linewidth]{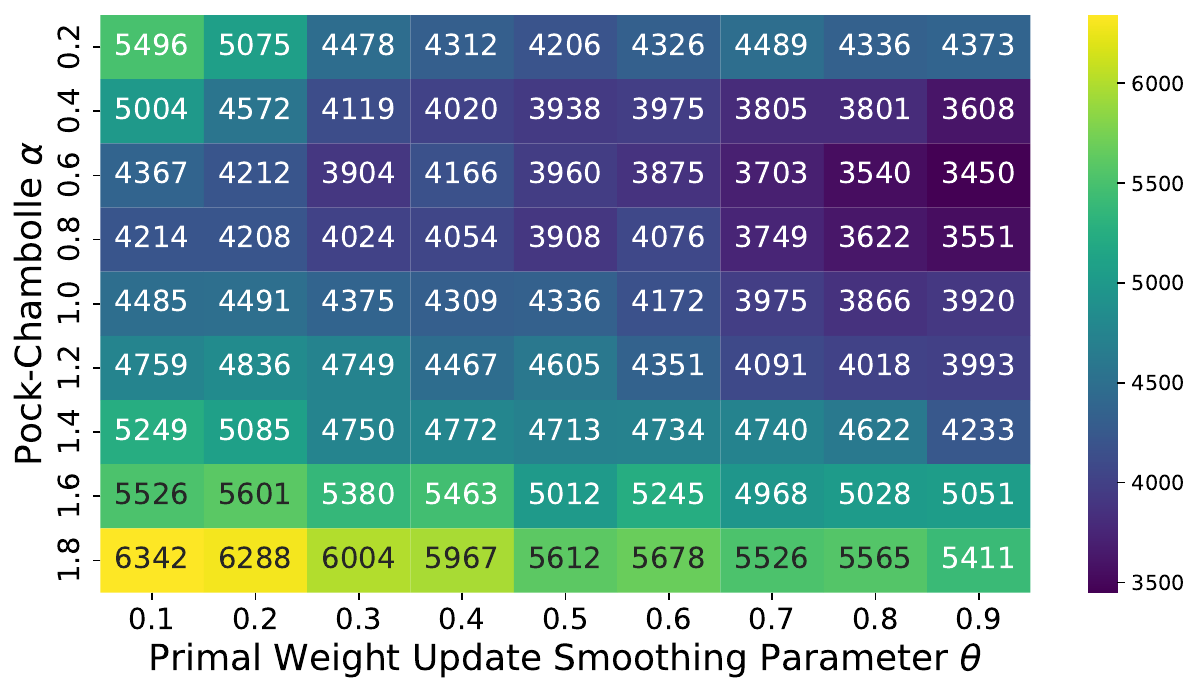}
    \includegraphics[trim={0 0 0 0},clip,width=0.495\linewidth]{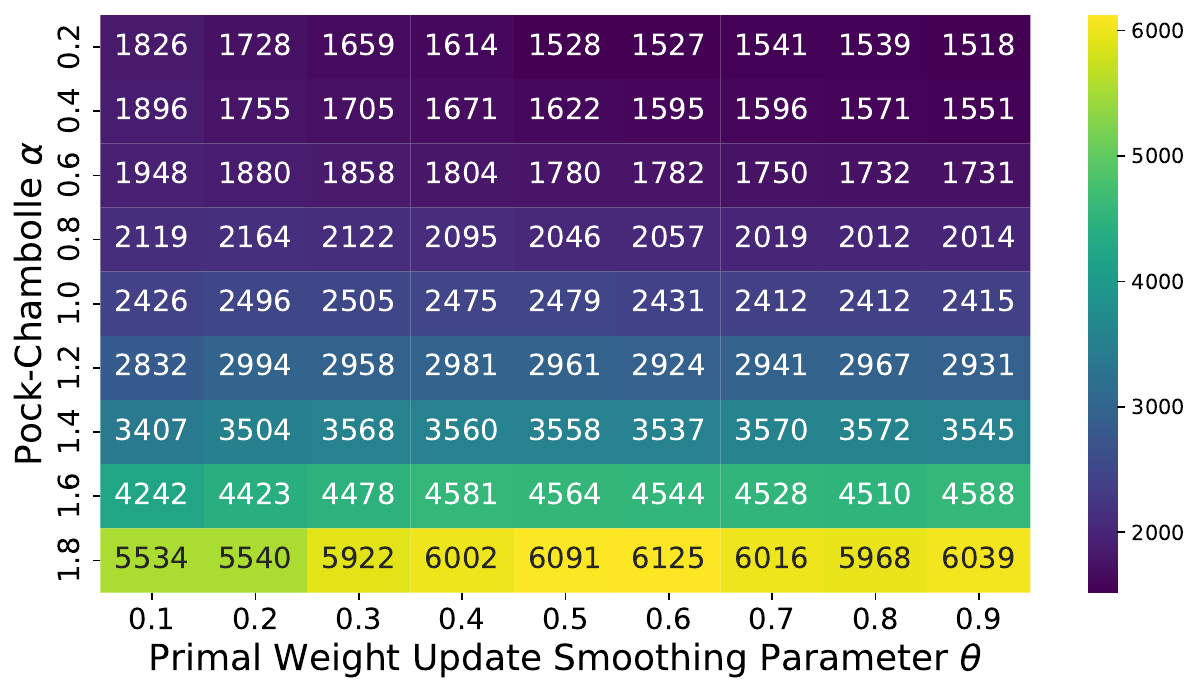}
    \includegraphics[trim={0 0 0 0},clip,width=0.495\linewidth]{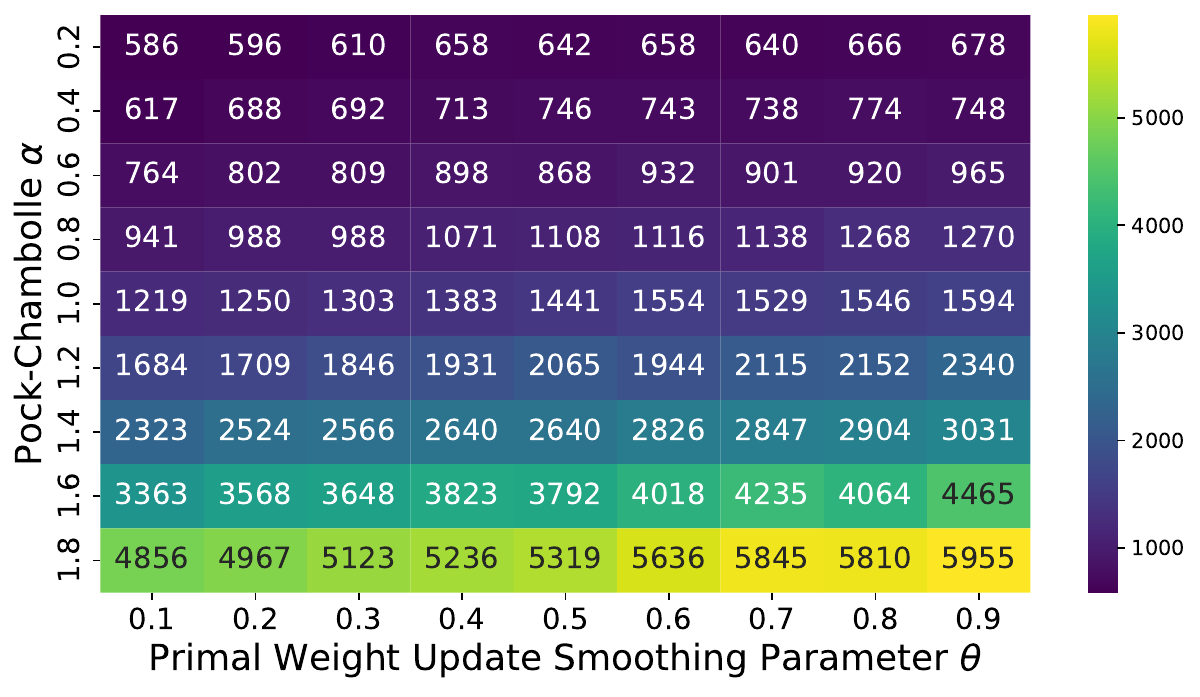}
    \caption{Grid search over $(\theta,\alpha)$; iteration count (averaged over 100 instances). {\bf (Top left)} Transportation, 30 suppliers, 45 customers; {\bf (Top right)} Transportation, 60 suppliers, 90 customers; {\bf (Bottom left)} Perturbed QAP15; {\bf (Bottom right)} Decay-decay auction, 200 items, 2000 bids.}
    \label{fig:main_heatmaps}
\end{figure}

Figure~\ref{fig:theta_iteration_plots} displays the effect of tuning the primal weight update smoothing parameter $\theta$, fixing the Pock-Chambolle preconditioning parameter to its default value of $\alpha = 1$, for four of the instance distributions (complete plots are in Appendix~\ref{app:exps}). The ERM parameter for the decay-decay auction and the perturbed QAP15 instances is $\theta = 0.02$; both distributions following the general trend that smaller $\theta$ improves performance (though performance on QAP15 is worst around $\theta = 0.3$ and improves beyond that). Small values of $\theta$ lead to the worst performance for the transportation instances with performance improving as $\theta$ approaches $1$. The ERM $\theta$ for the multipaths auction is $0.2$, with no obvious trends suggesting how $\theta$ should be tuned when $\alpha = 1$ is fixed.

Figure~\ref{fig:alpha_iteration_plots} displays the effect of tuning the Pock-Chambolle preconditioning parameter $\alpha$, fixing the primal weight update smoothing parameter to its default value of $\theta = 0.5$, for two of the instance distributions (complete plots are in Appendix~\ref{app:exps}). Corresponding plots for the remaining instance distributions, included in Appendix~\ref{app:exps}, display the same trend seen for the decay-decay auction distribution. In general, smaller values of $\alpha$ appear preferable, and this trend is especially pronounced for the auction instances. The ERM $\alpha$ for the transportation instances is around $0.5$. In both cases, poor $\alpha$ choices lead to large blowups in iteration count (over $10\times$ for the auction instances).

Figure~\ref{fig:main_heatmaps} displays the effect of simultaneously tuning $\theta$ and $\alpha$, for four of the instance distributions (complete grid search plots are in Appendix~\ref{app:exps}). The regions of parameter space leading to best performance differ substantially across the four distributions. Poor parameter choices correspond to large (an order of magnitude in some cases) degradations in performance. For the displayed instance distributions, ERM parameters improve over the average iteration count of default parameters $(\theta,\alpha) = (0.5, 1)$ by a factor of (in clockwise order, starting from the top left; Figure~\ref{fig:main_heatmaps}) $1.12\times$, $1.25\times$, $1.63\times$, and $2.45\times$. (The most dramatic improvements by ERM over default parameters are with the auction instances.) A finer grid search is likely to yield even more dramatic improvements.

The iteration count trends and resulting ERM parameters are mirrored by the run-time trends. Full run-time plots are in Appendix~\ref{app:exps}.

\section{Conclusions and Future Research}

We established generalization guarantees for tuning the most important hyperparameters within (cu)PDLP, a new state-of-the-art first order method for linear programming that runs on GPUs. First, we showed that the step size and primal weight of the vanilla PDHG algorithm---which underpins PDLP---are learnable with linear sample complexity. Our analysis showed that the primal and dual trajectories were piecewise rational functions with bounded degree. Next, we characterized the behavior of PDLP trajectories and leveraged the recent theory of data-driven algorithm design for algorithms with {\em Pfaffian} structure to obtain polynomial sample complexity guarantees for PDLP hyperparameter tuning. Finally, we presented experiments demonstrating the importance of hyperparameter tuning, and that distribution-dependent tuning is essential, with significant performance gains over default parameter settings.\looseness-1

There are many fruitful directions for future research. First, we believe our results can be directly extended to yield a more unified generalization theory for first-order PDHG-based LP methods including the gamut of algorithms proceeding PDLP~\citep{lu2024restarted,lu2024mpax,lu2025cupdlpx}. Results in that direction could shed light on automated design of higher-level algorithmic aspects beyond tuning of prespecified hyperparameters. Two important practical next steps are scaling up to larger-scale LP instances and {\em instance specific} parameter selection, where a (neural network based, say) model is trained to predict good parameter settings based on instance features. Guarantees along this line have been obtained in the related domain of integer programming~\citep{cheng2024sample}, and we expect that they are possible here as well. This direction has the potential for even greater practical performance improvements.

\bibliography{references}
\bibliographystyle{plainnat}
\newpage
\appendix
\section{Additional Related Work}\label{app:related_work}

\paragraph{First order methods for mathematical programming}  PDLP and its GPU-accelerated counterpart cuPDLP have been further improved with cuPDLPx~\citep{lu2025cupdlpx}, which uses a different averaging scheme based on Halpern PDHG~\citep{lu2024restarted} and a PID controller for primal-weight updates. MPAX~\citep{lu2024mpax} is a library that provides JAX implementations of (variants of) cuPDLP and cuPDLPx. Beyond linear programming, the promise of GPU-acclerated first order methods has been explored for large-scale quadratic programming~\citep{huang2025restarted,lu2026practical}.~\citet{lu2025overview} provide a comprehensive survey of GPU-based first-order methods for LP. We expect that the theory developed in the present paper can be adapted in a standard way to derive generalization guarantees for all aforementioned PDLP variants.

\paragraph{Data-driven algorithm design for numerical optimization}

There is a large body of work that studies automated configuration of branch-and-cut for integer programming, via the tools of data-driven algorithm design~\citep{balcan2021sample,balcan2022improved,balcan2022structural,cheng2024sample,cheng2025generalization}. Other applications include dimensionality-reduction methods for linear and quadratic programming~\citep{sakaue2024generalization,iwata2025learning,nguyen2026provably}, Lagrangian methods for integer programming~\citep{le2026provably}, and data-driven numerical linear algebra (in particular, algorithms for low-rank approximation)~\citep{bartlett2022generalization}.

The data-driven algorithm design toolkit that studies algorithm configuration through pseudodimension analysis has found wide-reaching applications. Use cases include auction design and pricing~\citep{ijcai2021p5,balcan2025increasing,balcan2025generalization}, computational biology ({\em e.g.,} sequence alignment algorithms)~\citep{balcan2024much}, and various machine learning algorithm configuration settings such as clustering~\citep{balcan2025algorithm}, regression algorithms~\citep{balcan2023new}, decision tree learning~\citep{balcan2024learning}, and neural network tuning~\citep{balcan2025sample}.

\section{Omitted PDLP Details}\label{app:pdlp}

We record here PDLP details omitted from the main body of the paper. This includes the adaptive PDHG step (Algorithm~\ref{alg:pdlp}, line~\ref{line:pdhg_step}) and the precise specifications of the restart and termination criteria, $\restart$ and $\terminate$.

\subsection{Termination}\label{app:termination}

Solution progress is measured via duality gap, primal feasibility, and dual feasibility {\em with respect to the original problem data}. Recall $\vec{r}_1$ and $\vec{r}_2$ collect the impacts of all scaling operations performed in preconditioning (Algorithm~\ref{alg:preconditioning}).
%
Consider the primal-dual iterates $\vec{x}^{s, t+1}, \vec{y}^{s, t+1}$ at the point Algorithm~\ref{alg:pdlp} enters the termination check at line~\ref{line:termination+restart}.
The {\em unscaled} iterates are: $\vec{r}_1^{-1} * \vec{x}^{s, t+1}$, $\vec{r}_2^{-1} * \vec{y}^{s, t+1}$ ($*$ denotes component-wise product). Let $\vec{c}$, $\vec{K} = (\vec{G}, \vec{A})$, $\vec{q} = (\vec{h},\vec{b})$, $\vec{\ell}$, $\vec{u}$ denote the {\em original problem data}. Then,
\begin{equation}
\begin{aligned}
\mathtt{gap}(\vec{x},\vec{y})&\coloneq\Abs{\vec{q}^\top\vec{y}+\vec{\ell}^\top\vec{\lambda}^{+}-\vec{u}^\top\vec{\lambda}^{-}-\vec{c}^\top\vec{x}}, \\
\mathtt{pfeas}(\vec{x},\vec{y}) &\coloneq \Norm{\begin{bmatrix}\vec{A}\vec{x}-\vec{b} \\ (\vec{h}-\vec{G}\vec{x})^+\end{bmatrix}}_2, \\ 
\mathtt{dfeas}(\vec{x},\vec{y}) &\coloneq \Norm{\vec{c}-\vec{K}^\top\vec{y}-\vec{\lambda}}_2.
\end{aligned}
\end{equation}
Given primal-dual pair $\vec{x},\vec{y}$ in the above, dual variable $\vec{\lambda}$---which corresponds to the lower/upper bound constraints $\vec{\ell}\le\vec{x}\le\vec{u}$---is computed as $$\vec{\lambda}=\proj_{\Lambda_1\times\cdots\times\Lambda_n}(\vec{c}-\vec{K}^\top\vec{y})$$ by clamping the coordinates of $\vec{c}-\vec{K}^\top\vec{y}$ to the appropriate range $\Lambda_i$.
Let $\terminate$ denote the boolean expression $$\begin{aligned}
\terminate(\vec{x},\vec{y}) &=
\mathbf{1}\left[\gap(\vec{x},\vec{y})\le\varepsilon_{\tol}(1+\abs{\vec{q}^\top\vec{y}+\vec{\ell}^\top\vec{\lambda}^+-\vec{u}^\top\vec{\lambda}^-}+\abs{\vec{c}^\top\vec{x}})\right] \\ &\qquad\wedge \mathbf{1}\left[\pfeas(\vec{x},\vec{y}) \le \varepsilon_{\tol}(1+\norm{\vec{q}}_2)\right] \wedge\mathbf{1}\left[\dfeas(\vec{x},\vec{y})\le \varepsilon_{\tol}(1+\norm{\vec{c}}_2)\right].
\end{aligned}$$ PDLP terminates when $\terminate$ returns $\mathtt{TRUE}$. 

\subsection{Restarts}\label{app:restarts}

KKT error is a progress metric that combines dual gap and primal-and-dual feasibility. It is used to determine when the inner loop is restarted from a new point and is defined by $$\kkt_{\omega}(\vec{x},\vec{y}) = \sqrt{\omega^2\Norm{\begin{bmatrix}\vec{A}\vec{x}-\vec{b} \\ (\vec{h}-\vec{G}\vec{x})^+\end{bmatrix}}_2^2 + \frac{\Norm{\vec{c}-\vec{K}^\top\vec{y}-\vec{\lambda}}_2^2}{\omega^2} + \Abs{\vec{q}^\top\vec{y}+\vec{\ell}^\top\vec{\lambda}^{+}-\vec{u}^\top\vec{\lambda}^{-}-\vec{c}^\top\vec{x}}^2}.$$ (We have not directly used the definitions of $\gap$, $\pfeas$ and $\dfeas$ within $\kkt$ since the former three metrics are always computed with respect to the original problem data. KKT error is computed with respect to the current iterates and preconditioned data. This includes computation of $\vec{\lambda}$.) 

The restart condition $\mathtt{res}(\widehat{\vec{x}}^{s,t+1},\widehat{\vec{y}}^{s,t+1},\widehat{\vec{x}}^{s,t},\widehat{\vec{y}}^{s,t}, \vec{x}^{s,0},\vec{y}^{s,0},\omega^s)$ is the OR of the following predicates
\begin{equation*}\begin{aligned} &\kkt_{\omega^s}(\widehat{\vec{x}}^{s, t+1}, \widehat{\vec{y}}^{s, t+1}) \le 0.2\cdot\kkt(\vec{x}^{s, 0}, \vec{y}^{s, 0}; \omega^s) \\ &\kkt_{\omega^s}(\widehat{\vec{x}}^{s, t+1}, \widehat{\vec{y}}^{s, t+1})\le 0.8 \cdot\kkt_{\omega^s}(\vec{x}^{s, 0}, \vec{y}^{s, 0})\wedge \kkt_{\omega^s}(\widehat{\vec{x}}^{s, t+1}, \widehat{\vec{y}}^{s, t+1}) > \kkt_{\omega^s}(\widehat{\vec{x}}^{s, t}, \widehat{\vec{y}}^{s, t}) \\ &t\le 0.36k\end{aligned}\end{equation*}

{\em For restarts, gap, primal feasibility, and dual feasibility are evaluated with respect to the preconditioned problem data.}

\subsection{Adaptive PDHG Step Algorithm}
Algorithm~\ref{alg:adaptivepdhg} is the line search routine developed by~\citet{applegate2021practical} to determine the step size used in the PDHG step for each iteration of the inner loop of PDLP. It dampens the candidate step size $\eta$ until the bound of line~\ref{line:apdhg_bound} is met. That bound comes from the convergence rate proof for PDHG~\citep{chambolle2016ergodic}.
\begin{algorithm}[h]
{\small
\caption{$\mathtt{adaptivePDHG}(\vec{x},\vec{y}, \omega, \widehat{\eta}, k)$}\label{alg:adaptivepdhg}
\begin{algorithmic}[1]
\State $\eta\leftarrow\widehat{\eta}$
    \While{$\mathtt{TRUE}$}
        \State $\vec{x}'\leftarrow\proj_X(\vec{x}-\tfrac{\eta}{\omega}(\vec{c} - \vec{K}^\top\vec{y}))$;~$\vec{y}'\leftarrow\proj_Y(\vec{y}+\eta\omega(\vec{q}-\vec{K}(2\vec{x}' - \vec{x})))$
        \State $\overline{\eta}\leftarrow\tfrac{\omega\norm{\vec{x'}-\vec{x}}_2^2 + (1/\omega)\norm{\vec{y}'-\vec{y}}_2^2}{2(\vec{y}'-\vec{y})^\top\vec{K}(\vec{x}'-\vec{x})}$ {\bf if} $2(\vec{y}'-\vec{y})^\top\vec{K}(\vec{x}'-\vec{x}) > 0$ {\bf else} $\infty$\label{line:apdhg_bound}
        \State $\eta'\leftarrow\min\{(1 - (k+1)^{-0.3})\overline{\eta}, (1 + (k+1)^{-0.6})\eta\}$
        \If{$\eta\le\overline{\eta}$}
            \State \Return $(\vec{x}', \vec{y}'), \eta, \eta'$
        \EndIf
        \State $\eta\leftarrow\eta'$
    \EndWhile
\end{algorithmic}
}
\end{algorithm}

The main loop in Algorithm~\ref{alg:adaptivepdhg} does not have an explicit bound on the number of iterations. However, we can establish the following bound that is dependent on the input step size $\widehat{\eta}$ and the current value of the total iteration counter $k$.

\begin{lemma}\label{lemma:adaptive-pdhg-iter-bound}
    Assume $\norm{\vec{K}}_{\infty}\ge 1$.
    The adaptive PDHG step algorithm (Algorithm~\ref{alg:adaptivepdhg}) terminates in at most $$\frac{\log(\widehat{\eta}\norm{\vec{K}}_2)}{\log(1/(1 - (k+1)^{-0.3}))} = O\left(ST\frac{\log\norm{\vec{K}}_2}{\log(1/(1 - (k+1)^{-0.3}))}\right)$$ iterations.
\end{lemma}

\begin{proof}
    The key fact is that in Algorithm~\ref{alg:adaptivepdhg}, $\overline{\eta} \ge 1/\norm{\vec{K}}_2$ always holds~\citep{applegate2021practical}. So, an upper bound on the number of iterations of the main loop is the number of iterations until $\eta$ drops below $1/\norm{\vec{K}}_2$. Now, for each round that $\eta$ is rejected (that is, the if-statement evaluates to false), we have $\eta > \overline{\eta}$, so $(1 - (k+1)^{-0.3})\overline{\eta}\le (1 + (k+1)^{-0.6})\eta$. So, $\eta$ is updated to $(1 - (k+1)^{-0.3})\overline{\eta}\le (1 - (k+1)^{-0.3})\eta$, that is, it reduces by a factor $(1 - (k+1)^{-0.3})$ on each rejected round, which shows that there are at most $\log(\widehat{\eta}\norm{K}_2)/\log(1 / (1 - (k+1)^{-0.3}))$ iterations. The right-hand-side bound follows from the fact that Algorithm~\ref{alg:adaptivepdhg} is entered at most $ST$ times, and on a final loop iteration where $\eta$ is accepted, it in the worst case increases by a factor of $(1 + (k+1)^{-0.6})$. Naively upper bounding this quantity by $2$ yields the global upper bound $\widehat{\eta}\le \widehat{\eta}^{0,0} 2^{ST} = 2^{ST} / \norm{\vec{K}}_{\infty} < 2^{ST}$. Substitution yields the desired bound.
\end{proof}

Lemma~\ref{lemma:adaptive-pdhg-iter-bound} is critical to establishing bounded Pfaffian complexity of PDLP (Appendix~\ref{app:pdlp}).

\section{Additional Details from Section~\ref{sec:pdhg}}\label{app:pdhg}

We use the following result due to~\citet{warren1968lower}: for $r$ polynomials $p_1,\ldots, p_r\in\R[z_1,\ldots, z_k]$ each of degree at most $d$, the number of connected components of $\R^k\setminus\bigcup_{i=1}^r\{(z_1,\ldots, z_k) : p_i(z_1,\ldots, z_k) = 0\}$ is at most $(4edr/k)^k$. The zero set of a polynomial is called an {\em algebraic curve}. Finally, recall that a {\em rational function} is a quotient of polynomials, and the {\em degree} of a rational function is the maximum degree of the two polynomials that define it.

\begin{proof}[Counting argument in the proof of Lemma~\ref{lemma:vanilla_pdhg_iterates}]
    At round $t$ we add a total of $L+U$ curves for the primal update and $m_1$ curves for the dual update, all with degree at most $2t$. Since $K$ algebraic curves with degree $\Delta$ induce at most $O(K^2\Delta^2)$ connected components, for each step $t$ we get a refinement of size at most $O(((L+U+m_1)t)^2(2t)^2)$. This is because, for any connected component induced by the algebraic curves up until round $t$, the cell is partitioned into smaller cells with boundaries given by at most $(L+U+m_1)t$ distinct algebraic curves. Across all rounds we therefore have at most $O(2^{2T}((L+U+m_1)^{2T}T^{4T})$ algebraic curves with degree at most $2T$ such that within each connected component they induce, $x_i^{(t)}$ and $y_j^{(t)}$ are rational functions in $\eta,\omega$ of degree at most $2t$, for all $i, j, t$. 
\end{proof}

\section{Proofs and additional details from Section \ref{sec:pdlp}}\label{app:pdlp}

\subsection{Proof of Theorem \ref{thm:pdim-theta}}
\begin{proof}
Fix an instance $x\in\mathcal X$, and consider the dual utility function
\[
u_x^*:\Theta\to\mathbb{Z},\qquad u_x^*(\theta)=u_\theta(x).
\]
We show that $u_x^*$ is piecewise Pfaffian with the claimed quantitative bounds.

Under the truncation parameters $(S,T)$ and the backtracking bound $B$ on the number of iterations of the adaptive line search in Algorithm 3,
the full execution of Algorithm~1 performs only finitely many scalar predicates.
We enumerate them exactly.

\paragraph{Predicates inside one backtracking round of Algorithm~3.}
Fix one call to Algorithm~3.
Each execution of the backtracking loop contributes:
\begin{itemize}
    \item $L+U$ primal projection predicates for the coordinates of
    \[
    x'=\operatorname{proj}_X\!\left(x-\frac{\eta}{\omega}(c-K^\top y)\right);
    \]
    \item $m_1$ dual projection predicates for the first block of
    \[
    y'=\operatorname{proj}_Y\!\left(y+\eta\omega(q-K(2x'-x))\right);
    \]
    \item one sign predicate for the denominator test
    \[
    2(y'-y)^\top K(x'-x)>0
    \]
    in line~4;
    \item one predicate for the branch in the minimum
    \[
    \eta'=\min\{(1-(k+1)^{-0.3})\eta,\ (1+(k+1)^{-0.6})\eta\};
    \]
    \item one acceptance predicate for the test in line~6.
\end{itemize}
Hence one backtracking round contributes exactly
\[
L+U+m_1+3
\]
scalar predicates.

Since each call to Algorithm~3 uses at most $B$ backtracking rounds, one call contributes at most
\[
B(L+U+m_1+3)
\]
predicates.

\paragraph{Predicates outside Algorithm~3 but inside one PDLP inner iteration.}
After the adaptive step returns, the PDLP inner iteration contributes:
\begin{itemize}
    \item one KKT comparison in line~7;
    \item three scalar predicates from the exact termination criterion $\mathrm{term}$ in Appendix~\ref{app:termination};
    \item three scalar predicates from the exact restart criterion $\mathrm{res}$ in Appendix~\ref{app:restarts}.
\end{itemize}
Hence the non-line-search part of one inner iteration contributes exactly
\[
1+3+3=7
\]
predicates.

Therefore one full PDLP inner iteration contributes at most
\[
B(L+U+m_1+3)+7
\]
predicates.

\paragraph{Outer-loop predicates.}
At each outer update, line~12 tests whether
\[
\|y_{s,0}-y_{s-1,0}\|_2>10^{-12}
\quad\text{and}\quad
\|x_{s,0}-x_{s-1,0}\|_2>10^{-12},
\]
which contributes two scalar predicates per outer iteration.
Across at most $S$ outer iterations this contributes $2S$ predicates.

Thus the total number of scalar predicate functions that can arise in the truncated execution is at most
\[
K_\theta^{\mathrm{full}}
=
ST\bigl(B(L+U+m_1+3)+7\bigr)+2S.
\]

From Lemma~\ref{lemma:adaptive-pdhg-iter-bound}, we can plug in $B = O((ST)^{1.3}\tilde{K})$.

\paragraph{Branch partition.}
Each predicate can be written as
\[
\tau(\theta)\ge 0
\quad\text{or}\quad
\tau(\theta)>0
\]
for some scalar function $\tau$ built from current intermediate quantities.
Let $\Xi_x$ be the union of all zeros of all such predicate functions.
Since only finitely many predicates occur, $\Xi_x$ is finite.
Hence $\Theta\setminus \Xi_x$ is a finite union of open intervals, and on each connected component
$I\subseteq \Theta\setminus\Xi_x$ every predicate has fixed sign.
Therefore the full truncated execution, including the exact line-search loop, follows a fixed branch
pattern on each such interval $I$. In particular, the acceptance predicate $\eta\le\overline{\eta}$ in Algorithm 3, line 6, is among the boundary functions. Its sign being fixed on $I$ means the number of backtracking iterations is also fixed on $I$, and hence the loop terminates at a fixed iteration count across the interval.

\paragraph{Pfaffianity on one branch interval.}
Fix one such interval $I$.
We prove by induction over the chronological execution order that every intermediate scalar quantity
computed by the algorithm is Pfaffian in $\theta$ on $I$.

The base case is immediate: the initial data, the preconditioned instance (with $\alpha$ fixed), the
initial step-size $\widehat\eta_{0,0}$, and the initial primal weight $\omega_0$ are constants in $\theta$.

For the inductive step, all updates inside Algorithm~3 are obtained from previously computed
Pfaffian quantities by arithmetic operations, divisions on regions where denominators are nonzero,
norms, inner products, and fixed branch selections. Therefore the quantities produced by each
backtracking round remain Pfaffian. The same holds for the averaging step, the KKT comparison,
the exact termination and restart quantities, and the line~12 primal-weight update
\[
\omega_s
=
\exp\!\left(
\theta\log\frac{\|y_{s,0}-y_{s-1,0}\|_2}{\|x_{s,0}-x_{s-1,0}\|_2}
+(1-\theta)\log \omega_{s-1}
\right)
\]
on the active branch, while on the inactive branch $\omega_s=\omega_{s-1}$.
Hence all intermediate quantities, and in particular the final utility value $u_x^*(\theta)$, are
Pfaffian on $I$.

Thus $u_x^*$ is piecewise Pfaffian.

\paragraph{Pfaffian piecewise-structure parameters.}
Take as boundary functions all scalar predicate functions appearing in the truncated execution.
Then
\[
k_G\le K_\theta^{\mathrm{full}}.
\]
A full sign pattern of these predicates determines at most one execution formula and hence at most
one Pfaffian piece function, so
\[
k_F\le 2^{K_\theta^{\mathrm{full}}}.
\]

Each backtracking round introduces only $O(1)$ new reciprocal or auxiliary functions, and each
outer update introduces only $O(1)$ new logarithmic/exponential functions from line~12.
Since there are at most $STB$ backtracking rounds in total, the chain length on each
piece satisfies
\[
q=O(STB).
\]
Using the same conservative bookkeeping as in the earlier $\theta$-analysis, the maximal Pfaffian
degree and the maximal degree of the functions arising on a piece satisfy
\[
M,\Delta\le 2^{O(STB)}.
\]
Hence the dual class $U_\Theta^*$ is $(k_F,k_G,q,M,\Delta,1)$-Pfaffian piecewise structured.

Theorem~\ref{thm:pfaffian-pdim} with parameter dimension $d=1$ yields
\[
\operatorname{Pdim}(U_\Theta)
\le
q^2+2q\log(\Delta+M)+2\log\!\bigl(\Delta(k_F+k_G)\bigr)+16.
\]
Substituting the bounds above gives
\[
\operatorname{Pdim}(U_\Theta)
=
O\!\Big((STB)^2 + K_\theta^{\mathrm{full}}\Big).
\]
Finally, substituting the definition of $K_\theta^{\mathrm{full}}$ and the bound on $B$ in Lemma \ref{lemma:adaptive-pdhg-iter-bound} yields the claimed bound.
\end{proof}

\subsection{Proof of Theorem~\ref{thm:pdim-alpha-formal}}

\begin{proof}
Fix an instance
\[
x=((x^{\mathrm{init}},y^{\mathrm{init}});c,G,h,A,b,\ell,u)\in\mathcal X.
\]
We prove that the dual utility function
\[
u_x^*:A\to\mathbb{Z},\qquad u_x^*(\alpha)=u_\alpha(x),
\]
is piecewise Pfaffian with the claimed quantitative bounds. Let $B$ be a bound on the number of iterations of the adaptive line search in Algorithm 3

\paragraph{Preconditioning phase.}
Algorithm~2 first performs ten Ruiz-scaling iterations, which are independent of $\alpha$.
After these ten iterations, let $\bar K$ denote the resulting matrix, and let
\[
r_j:=\|\bar K_{j,\cdot}\|_2>0\quad (j=1,\dots,m),\qquad
c_i:=\|\bar K_{\cdot,i}\|_2>0\quad (i=1,\dots,n),
\]
where positivity is guaranteed by assumption.

The final Pock--Chambolle scaling is
\[
D_1(\alpha)=\operatorname{diag}(r_1^{\,2-\alpha},\dots,r_m^{\,2-\alpha}),
\qquad
D_2(\alpha)=\operatorname{diag}(c_1^{\,\alpha},\dots,c_n^{\,\alpha}).
\]
Define
\[
\rho_j(\alpha):=r_j^{\,2-\alpha},\qquad
\gamma_i(\alpha):=c_i^{\,\alpha},\qquad
\delta_i(\alpha):=c_i^{-\alpha}.
\]
Each of these is a one-variable Pfaffian function of $\alpha$, since
\[
r_j^{\,2-\alpha}=r_j^2 e^{-(\log r_j)\alpha},
\qquad
c_i^\alpha=e^{(\log c_i)\alpha},
\qquad
c_i^{-\alpha}=e^{-(\log c_i)\alpha}.
\]
Therefore all entries of the preconditioned data
\[
\widetilde K(\alpha)=D_1(\alpha)\bar K D_2(\alpha),\qquad
\widetilde c(\alpha)=D_2(\alpha)\bar c,\qquad
\widetilde q(\alpha)^\top=\bar q^\top D_1(\alpha),
\]
\[
\widetilde u(\alpha)=D_2(\alpha)^{-1}\bar u,\qquad
\widetilde \ell(\alpha)=D_2(\alpha)^{-1}\bar \ell,
\]
as well as the accumulated scaling vectors $r_1(\alpha),r_2(\alpha)$ used to unscale the iterates in
the exact termination criterion, are Pfaffian in $\alpha$.
Hence the preconditioning phase contributes a Pfaffian chain of length
\[
O(n+m).
\]

\paragraph{Predicate count.}
Under the truncation parameters $(S,T)$ and the backtracking bound $B$, the exact
same predicate-counting argument as in the proof of Theorem~\ref{thm:pdim-theta} applies.
Each backtracking round in Algorithm~3 contributes
\[
L+U+m_1+3
\]
predicates, each full PDLP inner iteration contributes an additional $7$ predicates from the KKT,
termination, and restart checks, and the line~12 primal-weight update contributes $2S$
outer-level predicates in total.
Therefore the total number of scalar predicate functions is at most
\[
K_\alpha^{\mathrm{full}}
=
ST\bigl(B(L+U+m_1+3)+7\bigr)+2S.
\]

Note that $K_\alpha^{\mathrm{full}} = K_\theta^{\mathrm{full}}$; the predicate counts are identical since the same branching decisions arise regardless of which parameter is being tuned. What differs is the Pfaffian chain length, which is $O(n + m + STB)$ for $\alpha$ versus $O(STB)$ for $\theta$, reflecting the additional $\alpha$-dependence introduced by preconditioning.

\paragraph{Branch partition and Pfaffianity on each piece.}
Let $\Xi_x$ be the union of the zero sets of all scalar predicate functions arising in the truncated
execution on instance $x$.
Then $A\setminus \Xi_x$ is a finite union of open intervals, and on each connected component
$I\subseteq A\setminus \Xi_x$ every branch outcome is fixed.

Fix such an interval $I$.
We prove by induction over execution order that every intermediate scalar quantity is Pfaffian in
$\alpha$ on $I$.

The base case follows from the Pfaffianity of the preconditioned data and the initial quantities.
For the inductive step, all updates in Algorithm~3 and Algorithm~1 are obtained from already
Pfaffian quantities by arithmetic operations, division on regions of fixed denominator sign,
norms, inner products, logarithm, exponential, and fixed branch selection.
This covers:
\begin{itemize}
    \item the exact projected primal and dual updates in Algorithm~3;
    \item the denominator-update formula and the acceptance comparison in Algorithm~3;
    \item the weighted averaging step in line~6 of Algorithm~1;
    \item the exact KKT quantity from Appendix~\ref{app:restarts};
    \item the exact termination criterion $\mathrm{term}$, evaluated on unscaled iterates
    $(r_1^{-1}*x,r_2^{-1}*y)$;
    \item the exact restart criterion $\mathrm{res}$ from Appendix~\ref{app:restarts};
    \item the line~12 primal-weight update, which depends on $\alpha$ through the iterates.
\end{itemize}
Hence the entire execution trace, and thus the final utility value $u_x^*(\alpha)$, is Pfaffian on $I$.

Therefore $u_x^*$ is piecewise Pfaffian on $A$.

\paragraph{Pfaffian piecewise-structure parameters.}
Taking all scalar predicate functions as boundary functions gives
\[
k_G\le K_\alpha^{\mathrm{full}}.
\]
A full sign pattern determines at most one execution formula, hence
\[
k_F\le 2^{K_\alpha^{\mathrm{full}}}.
\]

The preconditioning phase contributes chain length $O(n+m)$, and the full truncated execution
contributes at most $O(STB)$ additional chain elements, so
\[
q=O(n+m+STB).
\]
Using the same conservative bookkeeping as before,
\[
M,\Delta\le 2^{O(STB)}.
\]
Thus the dual class $U_A^*$ is $(k_F,k_G,q,M,\Delta,1)$-Pfaffian piecewise structured.

Theorem~\ref{thm:pfaffian-pdim} with parameter dimension $d=1$ yields
\[
\operatorname{Pdim}(U_A)
\le
q^2+2q\log(\Delta+M)+2\log\!\bigl(\Delta(k_F+k_G)\bigr)+16.
\]
Substituting the displayed bounds gives
\[
\operatorname{Pdim}(U_A)
=
O\!\Big((n+m+STB)^2 + K_\alpha^{\mathrm{full}}\Big).
\]
Substituting the definition of $K_\alpha^{\mathrm{full}}$ and the bound on $B$ in Lemma \ref{lemma:adaptive-pdhg-iter-bound} yields the claimed bound.
\end{proof}

\subsection{Tuning both parameters simultaneously}\label{app:extensions}

\begin{theorem}[Pseudo-dimension bound for jointly tuning $(\theta,\alpha)$ in PDLP]
\label{thm:pdim-theta-alpha-full}
Fix truncation parameters $S,T\in\mathbb N$, and let $\mathcal X$ be a class of LP
instances together with initial primal-dual points. For each
$(\theta,\alpha)\in \Theta\times A\subseteq \mathbb R\times(0,2)$, let
\[
u_{\theta,\alpha}:\mathcal X\to\mathbb{Z}
\]
be a bounded utility function obtained by running the full truncated version of Algorithm~1,
including the exact restart and termination criteria from Appendix~\ref{app:pdlp} and the exact adaptive
line-search routine of Algorithm~3, with at most $S$ outer iterations and at most
$T$ inner iterations per outer iteration, where preconditioning uses Algorithm~2 with
parameter $\alpha$.

Assume:
\begin{enumerate}
    \item for every instance $x\in\mathcal X$, every row norm and column norm used in the final
    Pock--Chambolle scaling step of Algorithm~2 is strictly positive;
    \item there is a uniform bound $\widehat\eta_{\max}$ on the input step size $\widehat\eta$
    passed into each call to Algorithm~3 during the truncated execution (for example, the one prescribed by Lemma~\ref{lemma:adaptive-pdhg-iter-bound});
    \item writing $\widetilde K(\alpha)$ for the preconditioned matrix produced by Algorithm~2,
    the quantity
    \[
    \kappa_A:=\sup_{\alpha\in A}\|\widetilde K(\alpha)\|_2
    \]
    is finite.
\end{enumerate}

Define
\[
B_{\max}
:=
\left\lceil
\frac{
\log\!\big(\widehat\eta_{\max}\,\kappa_A\big)
}{
\log\!\Big(\frac{1}{1-(ST)^{-0.3}}\Big)
}
\right\rceil,
\]
and let
\[
U_{\Theta,A}
=
\{u_{\theta,\alpha}:\mathcal X\to\mathbb{Z}\mid (\theta,\alpha)\in\Theta\times A\}.
\]
Let $L=|\{i:\ell_i>-\infty\}|$, $U=|\{i:u_i<\infty\}|$, and $m:=m_1+m_2$, and define
\[
K_{\theta,\alpha}^{\mathrm{full}}
:=
ST\bigl(B_{\max}(L+U+m_1+3)+7\bigr)+2S.
\]

Then there exist parameters $q,M,\Delta,k_F,k_G$ such that the dual class
\[
U_{\Theta,A}^*
=
\{u_x^*:\Theta\times A\to\mathbb{Z}\mid x\in\mathcal X\},
\qquad
u_x^*(\theta,\alpha):=u_{\theta,\alpha}(x),
\]
is $(k_F,k_G,q,M,\Delta,2)$-Pfaffian piecewise structured, with
\[
q = O(n+m+STB_{\max}),
\qquad
M,\Delta \le 2^{O(STB_{\max})},
\]
\[
k_G\le K_{\theta,\alpha}^{\mathrm{full}},
\qquad
k_F\le 2^{K_{\theta,\alpha}^{\mathrm{full}}}.
\]
Consequently,
\[
\operatorname{Pdim}(U_{\Theta,A})
\le
4q^2
+4q\log(\Delta+M)
+8q\log 2
+4\log\!\bigl(\Delta(k_F+k_G)\bigr)
+32.
\]
In particular,
\[
\operatorname{Pdim}(U_{\Theta,A})
=
O\!\Big((n+m+STB_{\max})^2
+
K_{\theta,\alpha}^{\mathrm{full}}\Big),
\]
that is,
\[
\operatorname{Pdim}(U_{\Theta,A})
=
O\!\Big(
(n+m+STB_{\max})^2
+
STB_{\max}(L+U+m_1)
\Big).
\]
\end{theorem}

\begin{proof}
Fix an instance
\[
x=((x^{\mathrm{init}},y^{\mathrm{init}});c,G,h,A,b,\ell,u)\in\mathcal X,
\]
and consider the dual utility function
\[
u_x^*:\Theta\times A\to\mathbb{Z},\qquad u_x^*(\theta,\alpha)=u_{\theta,\alpha}(x).
\]
We show that $u_x^*$ is piecewise Pfaffian with the claimed quantitative bounds.

\paragraph{Step 1: Pfaffian dependence on $(\theta,\alpha)$.}
By the positivity assumption on the row and column norms in the final Pock--Chambolle scaling,
the preconditioning maps
\[
\alpha\mapsto D_1(\alpha),\qquad \alpha\mapsto D_2(\alpha)
\]
have diagonal entries of the form $r_j^{\,2-\alpha}$ and $c_i^\alpha$, hence are Pfaffian in
$\alpha$. Therefore all entries of the preconditioned data
\[
\widetilde K(\alpha),\ \widetilde c(\alpha),\ \widetilde q(\alpha),\ \widetilde \ell(\alpha),\
\widetilde u(\alpha),
\]
and the accumulated scaling vectors used to unscale the iterates in the exact termination criterion,
are Pfaffian in $\alpha$.

Once the preconditioned data are fixed, every update in Algorithm~3 and Algorithm~1 is built from
previously computed quantities using arithmetic operations, division on regions where denominators
are nonzero, norms, inner products, logarithm, exponential, and branch selection. The line~12
primal-weight update is the only place where $\theta$ appears explicitly:
\[
\omega_s
=
\exp\!\left(
\theta\log\frac{\|y_{s,0}-y_{s-1,0}\|_2}{\|x_{s,0}-x_{s-1,0}\|_2}
+
(1-\theta)\log\omega_{s-1}
\right)
\]
on the active branch, and $\omega_s=\omega_{s-1}$ on the inactive branch.
Hence, on every region where all branch outcomes are fixed, every intermediate quantity is Pfaffian
in the two parameters $(\theta,\alpha)$.

\paragraph{Step 2: exact predicate count.}
We now count the scalar predicates appearing in one truncated execution.

Each backtracking round of Algorithm~3 contributes:
\begin{itemize}
    \item $L+U$ primal projection predicates;
    \item $m_1$ dual projection predicates;
    \item one sign predicate for the denominator test in line~4;
    \item one predicate for the branch in the minimum in line~5;
    \item one acceptance predicate in line~6.
\end{itemize}
Thus one backtracking round contributes exactly
\[
L+U+m_1+3
\]
predicates.

By Lemma~\ref{lemma:adaptive-pdhg-iter-bound}, each call to Algorithm~3 performs at most
\[
\frac{\log(\widehat\eta\|\widetilde K(\alpha)\|_2)}
{\log\!\bigl(1/(1-(k+1)^{-0.3})\bigr)}
\]
backtracking iterations. Since throughout the truncated execution we have
$k+1\le ST$, $\widehat\eta\le\widehat\eta_{\max}$, and
$\|\widetilde K(\alpha)\|_2\le \kappa_A$, every call performs at most $B_{\max}$ backtracking
iterations.

After the adaptive step returns, each PDLP inner iteration contributes:
\begin{itemize}
    \item one KKT comparison in line~7;
    \item three scalar predicates from the exact termination criterion $\mathrm{term}$;
    \item three scalar predicates from the exact restart criterion $\mathrm{res}$.
\end{itemize}
Hence one full PDLP inner iteration contributes at most
\[
B_{\max}(L+U+m_1+3)+7
\]
predicates.

Across at most $ST$ inner iterations this gives
\[
ST\bigl(B_{\max}(L+U+m_1+3)+7\bigr).
\]
In addition, line~12 contributes two norm-threshold predicates per outer iteration, hence at most
$2S$ predicates in total.
Therefore the total number of scalar predicate functions is at most
\[
K_{\theta,\alpha}^{\mathrm{full}}
=
ST\bigl(B_{\max}(L+U+m_1+3)+7\bigr)+2S.
\]

\paragraph{Step 3: branch partition.}
Each predicate can be written as
\[
\tau(\theta,\alpha)\ge 0
\quad\text{or}\quad
\tau(\theta,\alpha)>0
\]
for some scalar function $\tau$ built from the current intermediate quantities.
Let $\Gamma_x$ denote the union of the zero sets of all such predicate functions.
Since only finitely many predicates occur, $\Gamma_x$ is a finite union of Pfaffian hypersurfaces in
$\Theta\times A$.
On each connected component of
\[
(\Theta\times A)\setminus \Gamma_x
\]
all branch outcomes are fixed, hence the entire truncated execution follows one fixed execution
formula on that component.

Therefore $u_x^*$ is piecewise Pfaffian on $\Theta\times A$.

\paragraph{Step 4: Pfaffian piecewise-structure parameters.}
Take all scalar predicate functions as boundary functions. Then
\[
k_G\le K_{\theta,\alpha}^{\mathrm{full}}.
\]
A full sign pattern of these predicates determines at most one execution formula for the truncated
trace and hence at most one Pfaffian piece function, so
\[
k_F\le 2^{K_{\theta,\alpha}^{\mathrm{full}}}.
\]

The preconditioning phase contributes a chain of length $O(n+m)$ from the $\alpha$-dependent
scaling functions. Each backtracking round contributes only $O(1)$ additional reciprocal or
auxiliary functions, and each outer update contributes only $O(1)$ additional logarithmic/exponential
functions from the $\theta$-update. Since there are at most $STB_{\max}$ backtracking
rounds in total,
\[
q=O(n+m+STB_{\max}).
\]
Using the same conservative bookkeeping as in the one-parameter analyses,
\[
M,\Delta\le 2^{O(STB_{\max})}.
\]
Thus the dual class $U_{\Theta,A}^*$ is
$(k_F,k_G,q,M,\Delta,2)$-Pfaffian piecewise structured.

\paragraph{Step 5: apply Theorem~\ref{thm:pfaffian-pdim}.}
Applying Theorem~\ref{thm:pfaffian-pdim} with parameter dimension $d=2$ yields
\[
\operatorname{Pdim}(U_{\Theta,A})
\le
4q^2
+4q\log(\Delta+M)
+8q\log 2
+4\log\!\bigl(\Delta(k_F+k_G)\bigr)
+32.
\]
Substituting the bounds above gives
\[
\operatorname{Pdim}(U_{\Theta,A})
=
O\!\Big((n+m+STB_{\max})^2
+
K_{\theta,\alpha}^{\mathrm{full}}\Big).
\]
Finally, substituting the definition of $K_{\theta,\alpha}^{\mathrm{full}}$ yields the claimed bound.
\end{proof}

\newpage

\section{Complete experimental results}\label{app:exps}

This appendix section contains complete the experimental results for the various LP instance distributions considered.

\subsection{Transportation instances}

Figure~\ref{fig:transport_iterations} displays the single-parameter tuning iteration count results; Figure~\ref{fig:transport_runtime} displays the single-parameter tuning run-time results; Figure~\ref{fig:transport_iterations_heatmap} displays the iteration count grid search results; Figure~\ref{fig:transport_runtime_heatmap} displays the run-time grid search results.

\begin{figure}[!htb]
    \centering
    \includegraphics[trim={0cm 0cm 0cm 0cm},clip, width=0.495\linewidth]{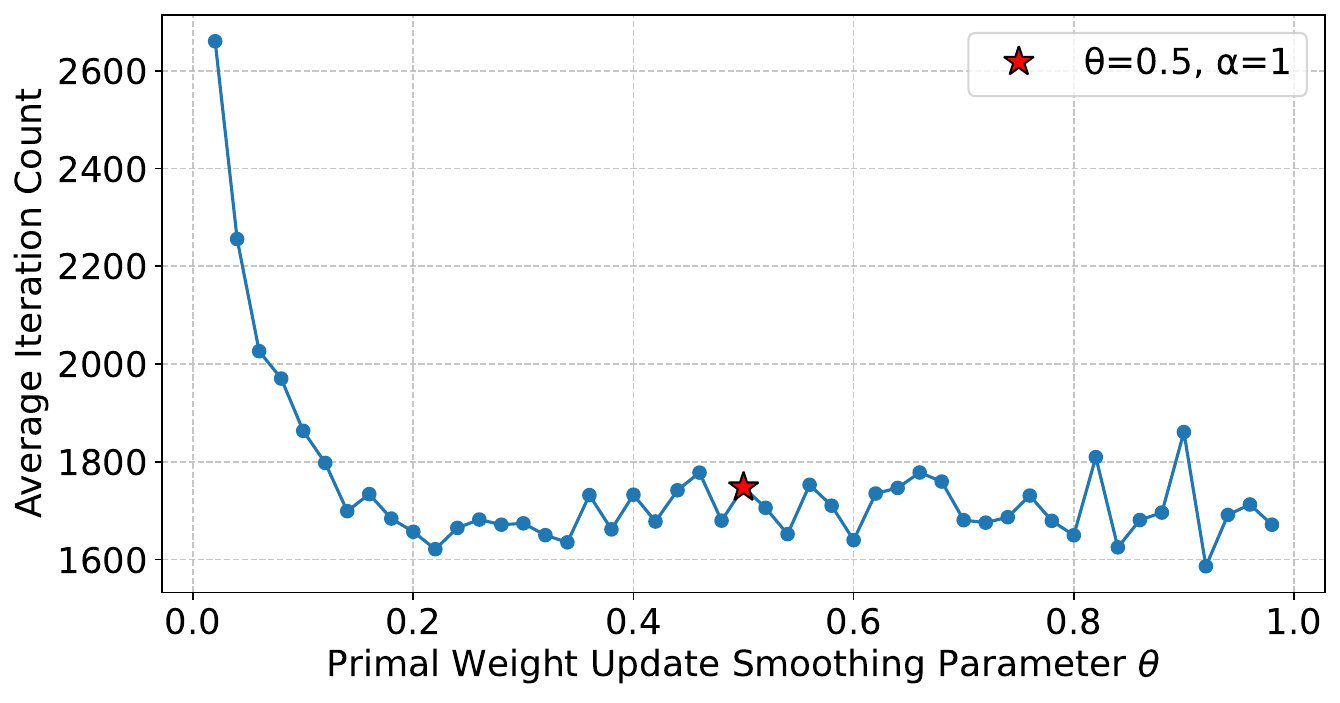}
    \includegraphics[trim={0 0 0 0},clip,width=0.495\linewidth]{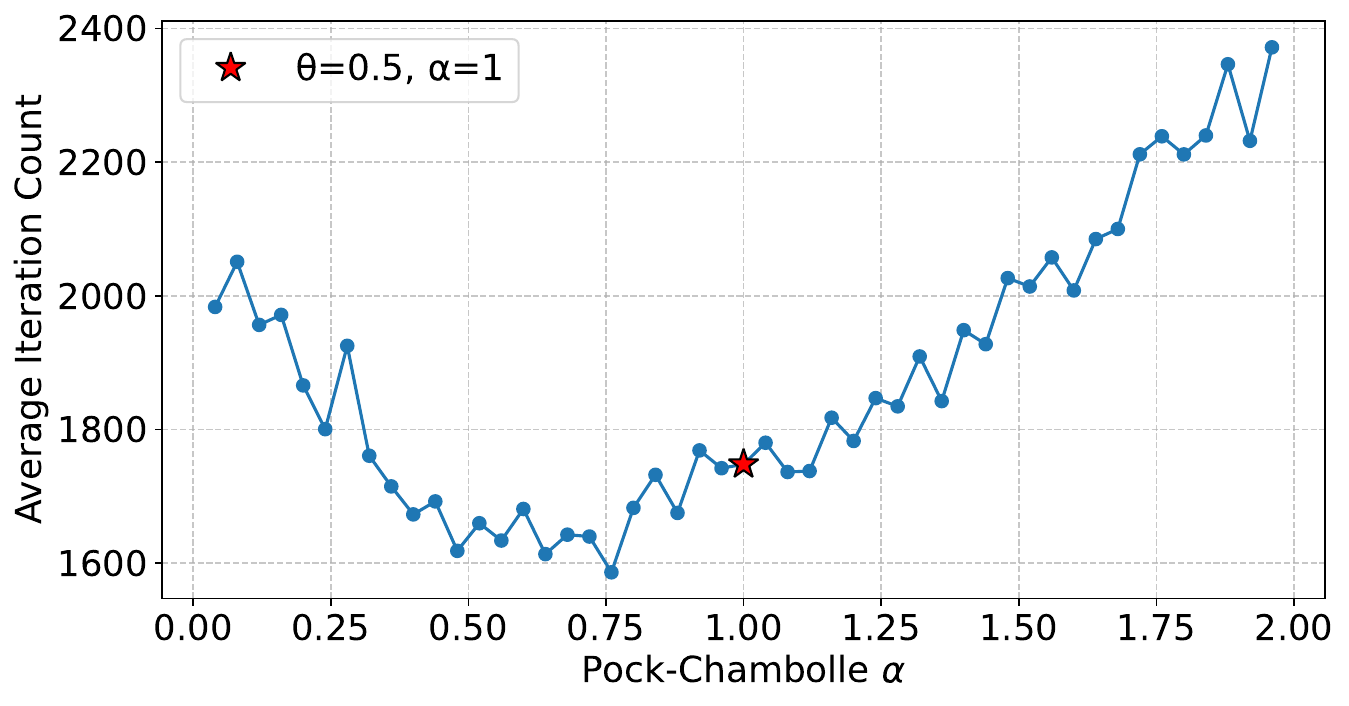}
    \includegraphics[trim={0cm 0cm 0cm 0cm},clip, width=0.495\linewidth]{figs/transport_60_90/iterations_vs_theta.pdf}
    \includegraphics[trim={0 0 0 0},clip,width=0.495\linewidth]{figs/transport_60_90/iterations_vs_alpha.pdf}
    \caption{Transportation iteration counts (averaged over 100 instances) with 30 suppliers and 45 customers (top), and 60 suppliers and 90 customers (bottom). {\bf (Left)} Effect of tuning $\theta$ with $\alpha = 1$. {\bf (Right)} Effect of tuning $\alpha$ with $\theta = 0.5$.}
    \label{fig:transport_iterations}
\end{figure}

\begin{figure}[!htb]
    \centering
    \includegraphics[trim={0cm 0cm 0cm 0cm},clip, width=0.495\linewidth]{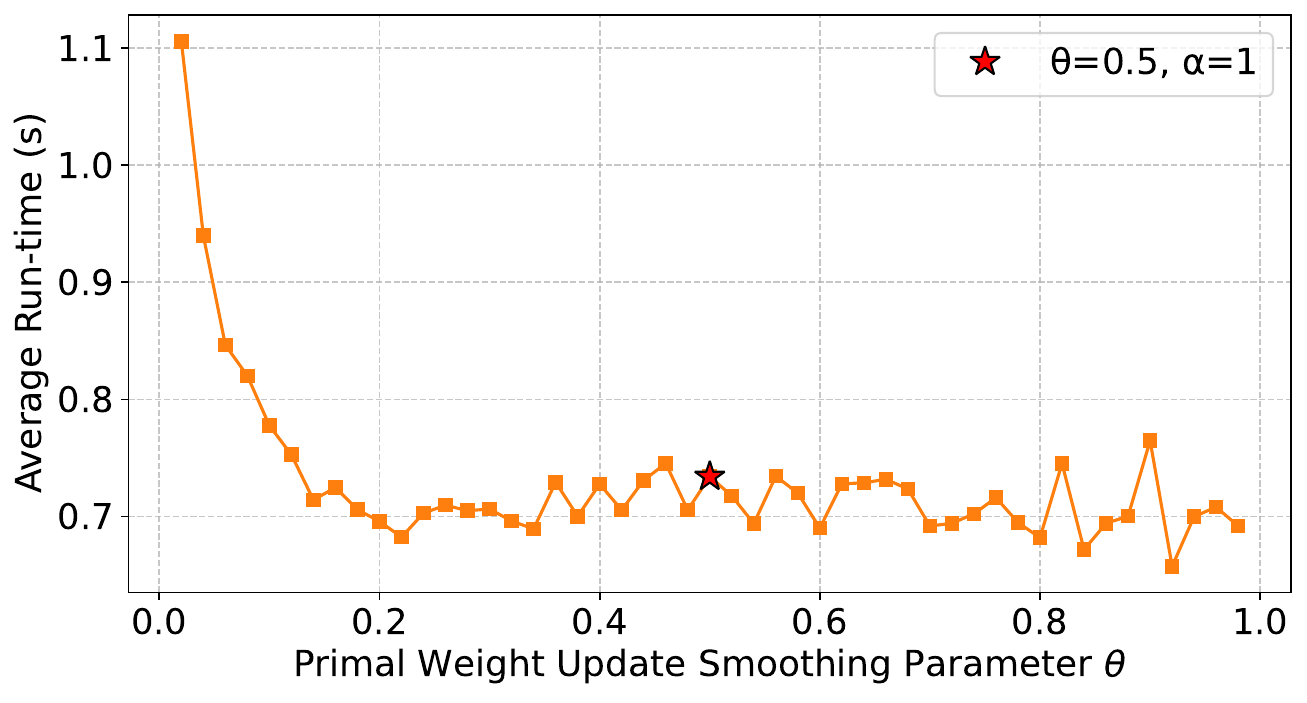}
    \includegraphics[trim={0 0 0 0},clip,width=0.495\linewidth]{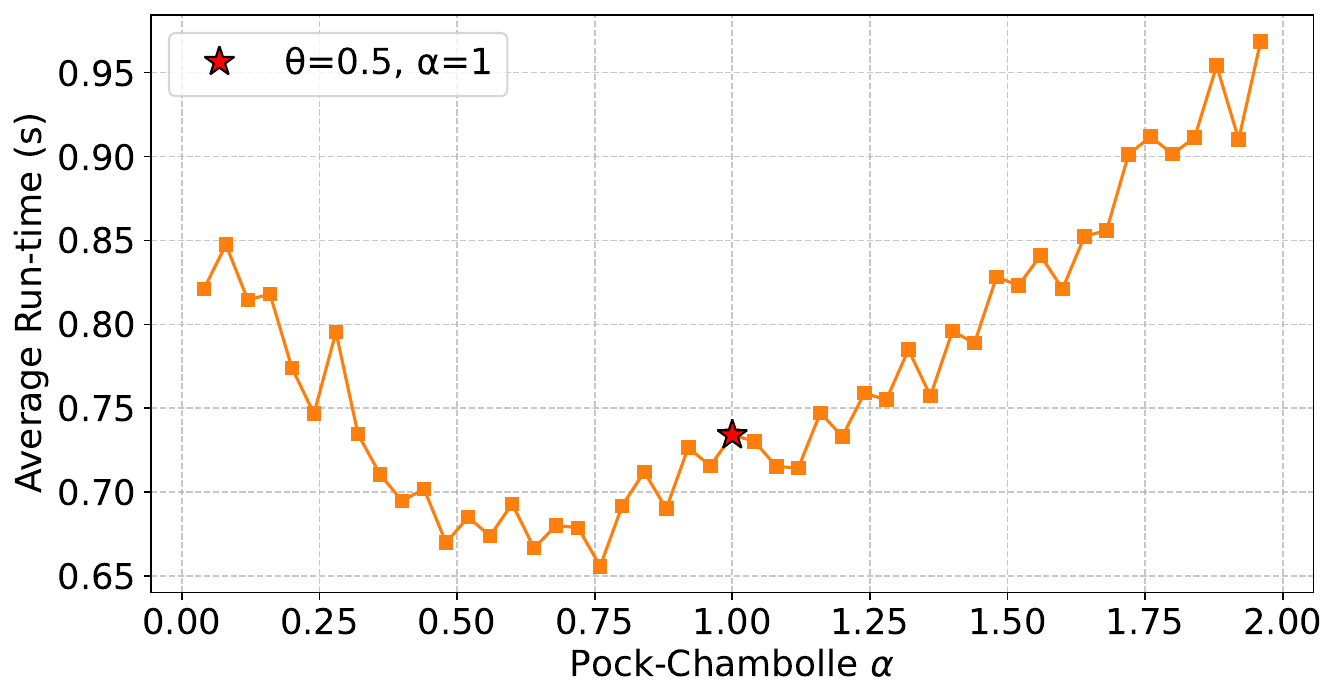}
    \includegraphics[trim={0cm 0cm 0cm 0cm},clip, width=0.495\linewidth]{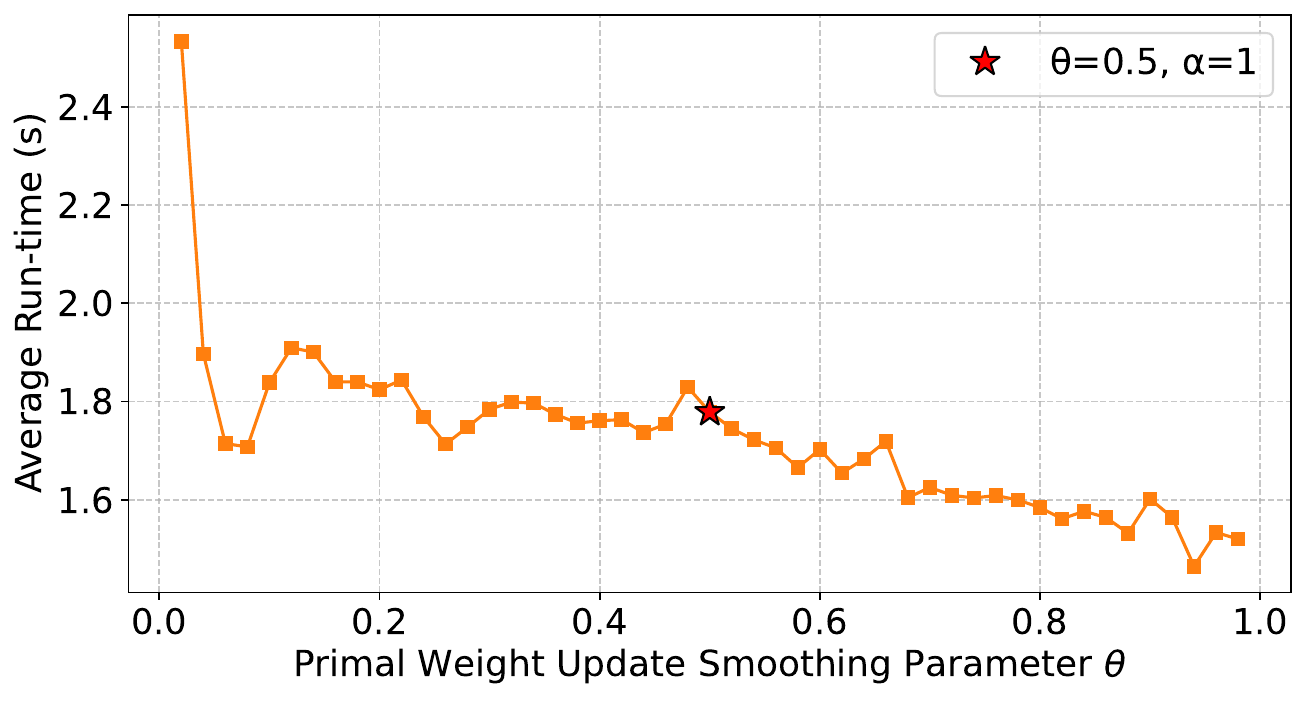}
    \includegraphics[trim={0 0 0 0},clip,width=0.495\linewidth]{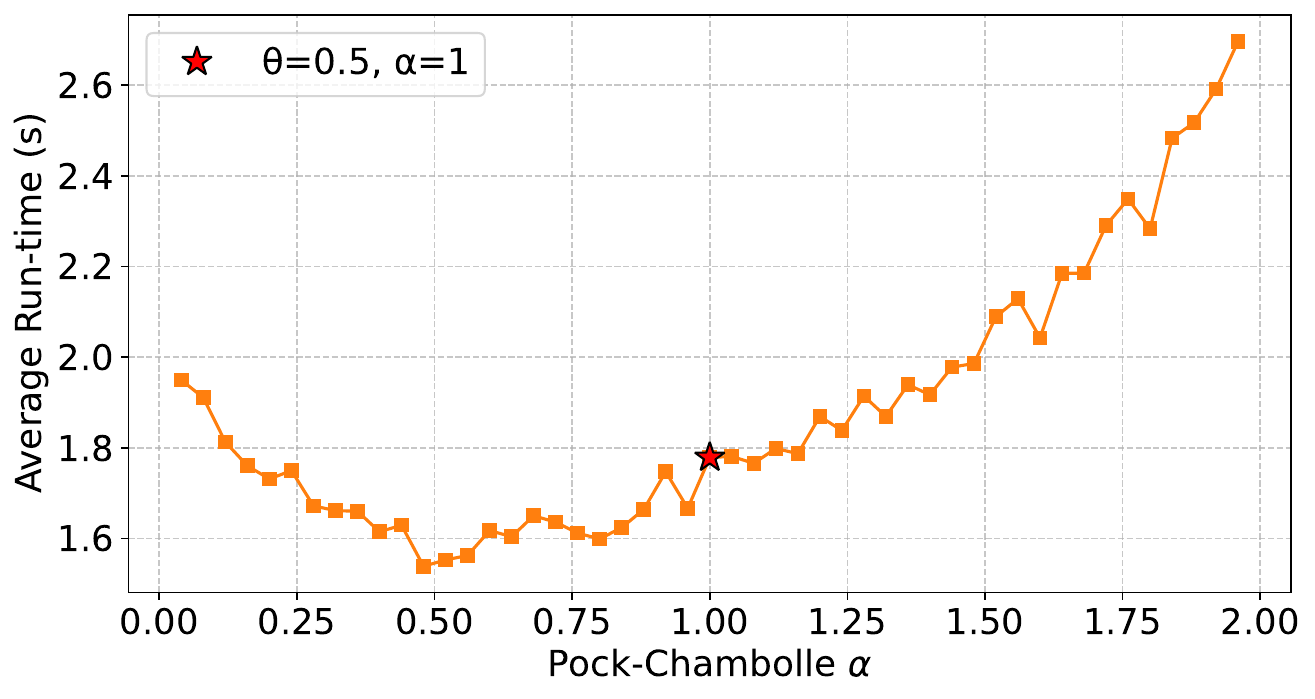}
    \caption{Transportation run-times (averaged over 100 instances) with 30 suppliers and 45 customers (top), and 60 suppliers and 90 customers (bottom). {\bf (Left)} Effect of tuning $\theta$ with $\alpha = 1$. {\bf (Right)} Effect of tuning $\alpha$ with $\theta = 0.5$.}
    \label{fig:transport_runtime}
\end{figure}

\begin{figure}[!htb]
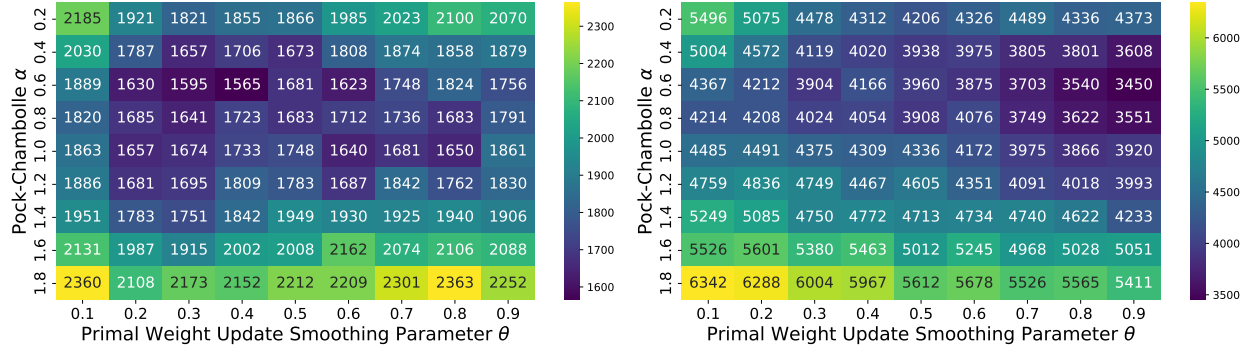

    \centering
    \includegraphics[trim={0 0 0 0},clip,width=0.495\linewidth]{figs/transport_30_45/grid_iterations_heatmap.pdf}
    \includegraphics[trim={0 0 0 0},clip,width=0.495\linewidth]{figs/transport_60_90/grid_iterations_heatmap.pdf}
    \caption{Transportation iteration counts (averaged over 100 instances); grid search over $(\theta,\alpha)$. {\bf (Left)} 30 suppliers, 45 customers; {\bf (Right)} 60 suppliers, 90 customers.}
    \label{fig:transport_iterations_heatmap}
\end{figure}

\begin{figure}[!htb]
    \centering
    \includegraphics[trim={0 0 0 0},clip,width=0.495\linewidth]{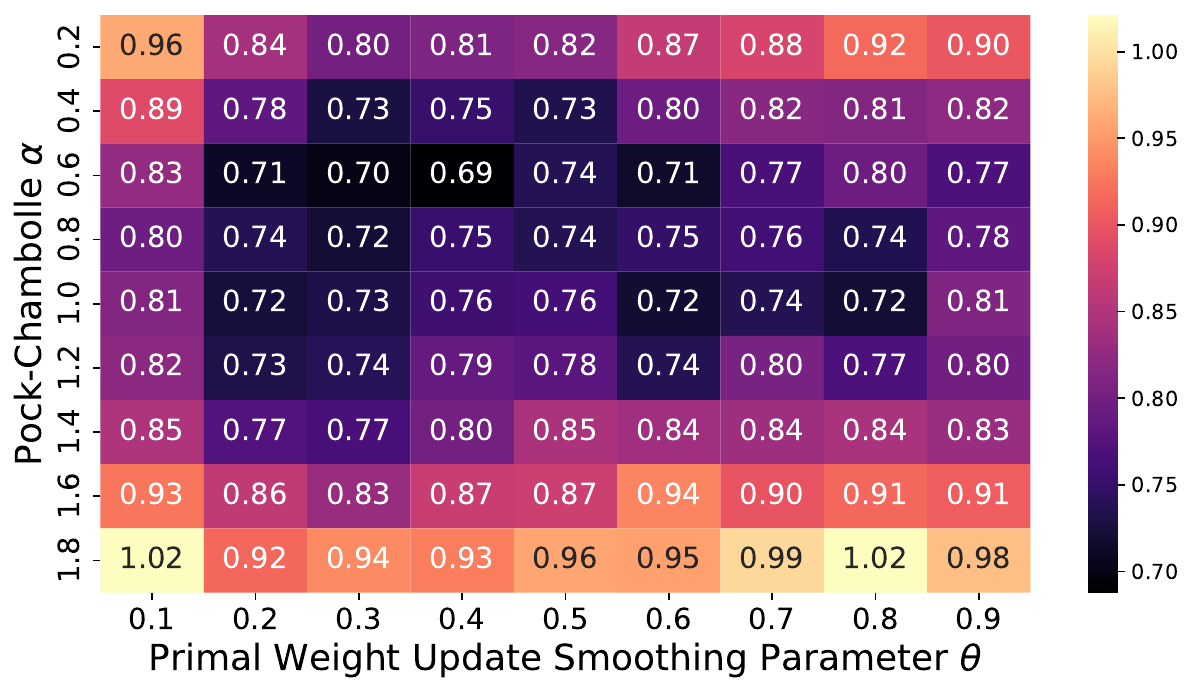}
    \includegraphics[trim={0 0 0 0},clip,width=0.495\linewidth]{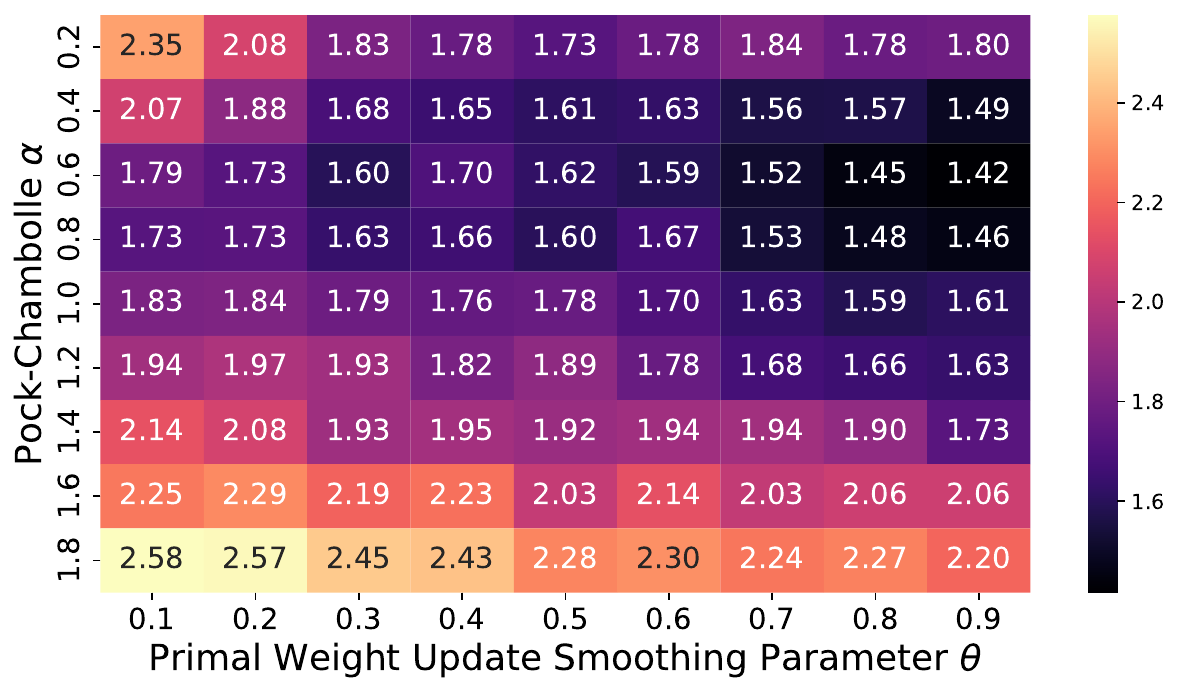}
    \caption{Transportation run-times (averaged over 100 instances); grid search over $(\theta,\alpha)$. {\bf (Left)} 30 suppliers, 45 customers; {\bf (Right)} 60 suppliers, 90 customers.}
    \label{fig:transport_runtime_heatmap}
\end{figure}

\newpage
\subsection{Perturbed QAP15 instances}

Figure~\ref{fig:qap15_iterations} displays the single-parmeter tuning iteration count results; Figure~\ref{fig:qap15_runtime} displays the single-parameter tuning run-time results; Figure~\ref{fig:qap15_heatmaps} displays the grid search results.

\begin{figure}[!htb]
    \centering
    \includegraphics[trim={0cm 0cm 0cm 0cm},clip, width=0.495\linewidth]{figs/qap15/iterations_vs_theta.pdf}
    \includegraphics[trim={0 0 0 0},clip,width=0.495\linewidth]{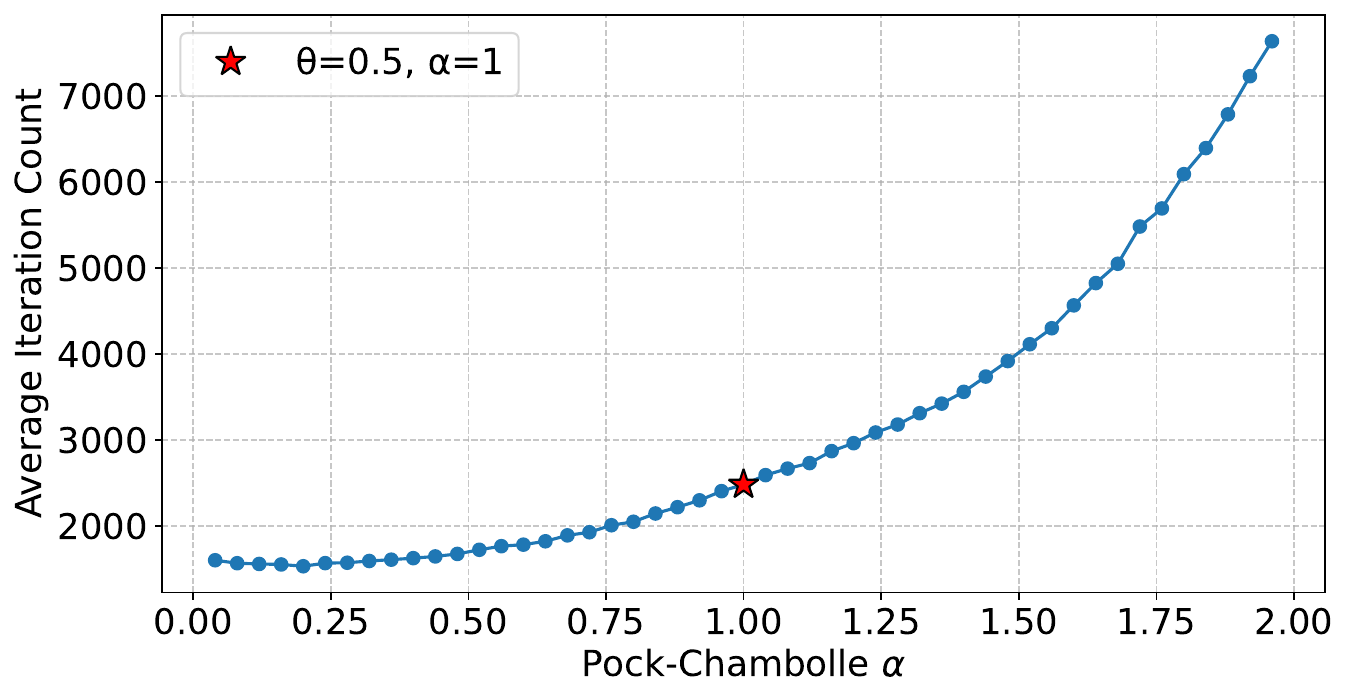}
    \caption{QAP15 iteration counts (averaged over 100 instances). {\bf (Left)} Effect of tuning $\theta$ with $\alpha = 1$. {\bf (Right)} Effect of tuning $\alpha$ with $\theta = 0.5$.}
    \label{fig:qap15_iterations}
\end{figure}

\begin{figure}[!htb]
    \centering
    \includegraphics[trim={0cm 0cm 0cm 0cm},clip, width=0.495\linewidth]{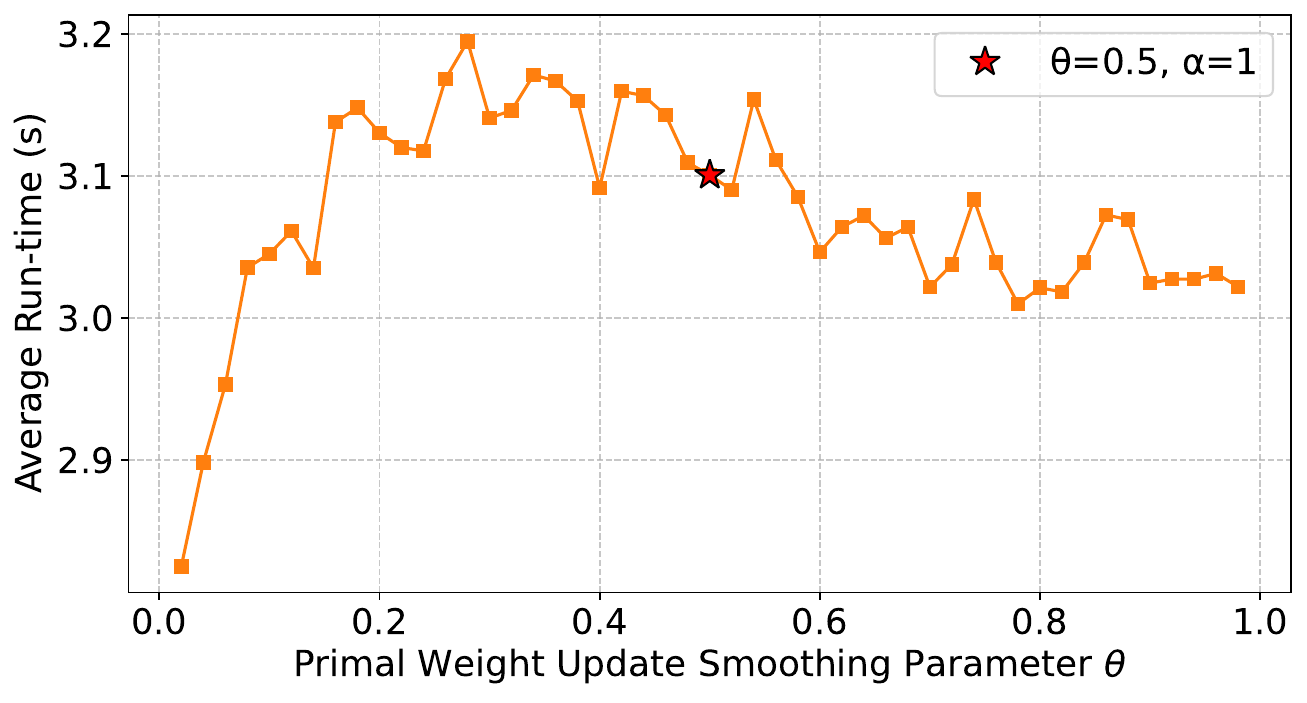}
    \includegraphics[trim={0 0 0 0},clip,width=0.495\linewidth]{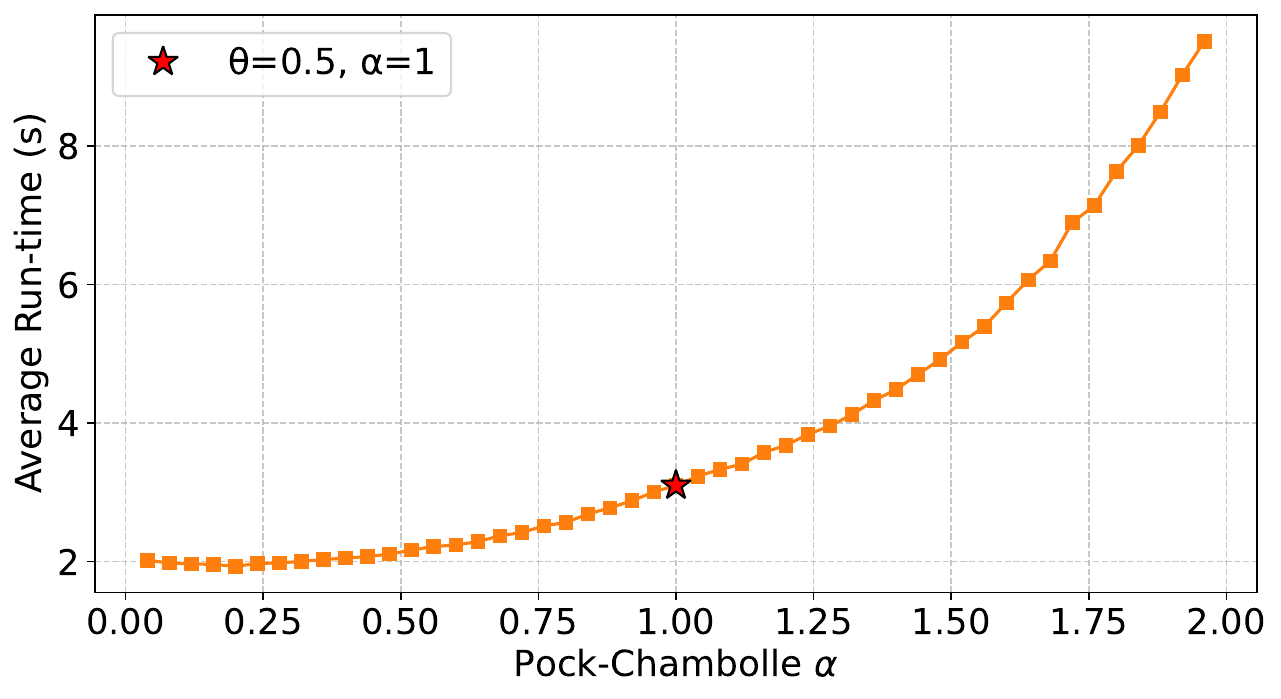}
    \caption{QAP15 run-times (averaged over 100 instances). {\bf (Left)} Effect of tuning $\theta$ with $\alpha = 1$. {\bf (Right)} Effect of tuning $\alpha$ with $\theta = 0.5$.}
    \label{fig:qap15_runtime}
\end{figure}

\begin{figure}[!htb]
    \centering
    \includegraphics[trim={0 0 0 0},clip,width=0.495\linewidth]{figs/qap15/grid_iterations_heatmap.pdf}
    \includegraphics[trim={0 0 0 0},clip,width=0.495\linewidth]{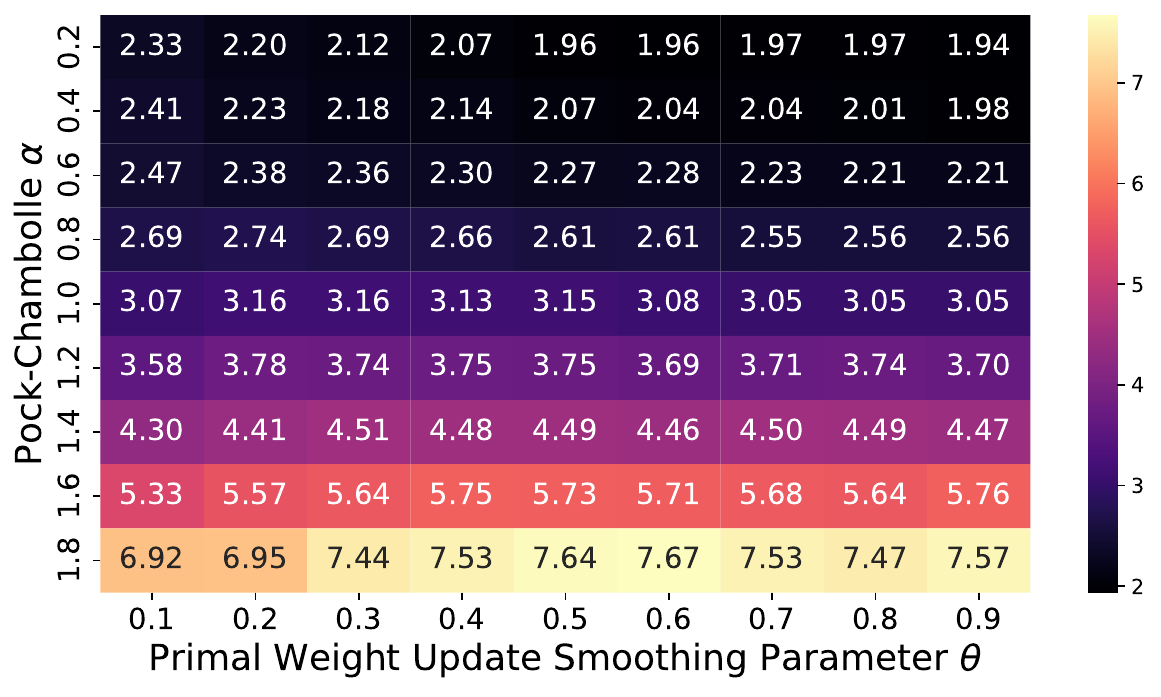}
    \caption{QAP15 iteration counts (left) and run-times (right); grid search over $(\theta,\alpha)$ (averaged over 100 instances).}
    \label{fig:qap15_heatmaps}
\end{figure}

\newpage
\subsection{Combinatorial auction instances}

\subsubsection{Decay-decay instances~\citep{sandholm2002winner}} Figure~\ref{fig:dd_auction_iterations} displays the single-parameter tuning iteration count results; Figure~\ref{fig:dd_auction_runtime} displays the single-parameter tuning run-time results; Figure~\ref{fig:dd_auction_iterations_heatmap} displays the iteration count grid search results; Figure~\ref{fig:dd_auction_runtime_heatmap} displays the run-time grid search results.

\begin{figure}[!htb]
    \centering
    \includegraphics[trim={0cm 0cm 0cm 0cm},clip, width=0.495\linewidth]{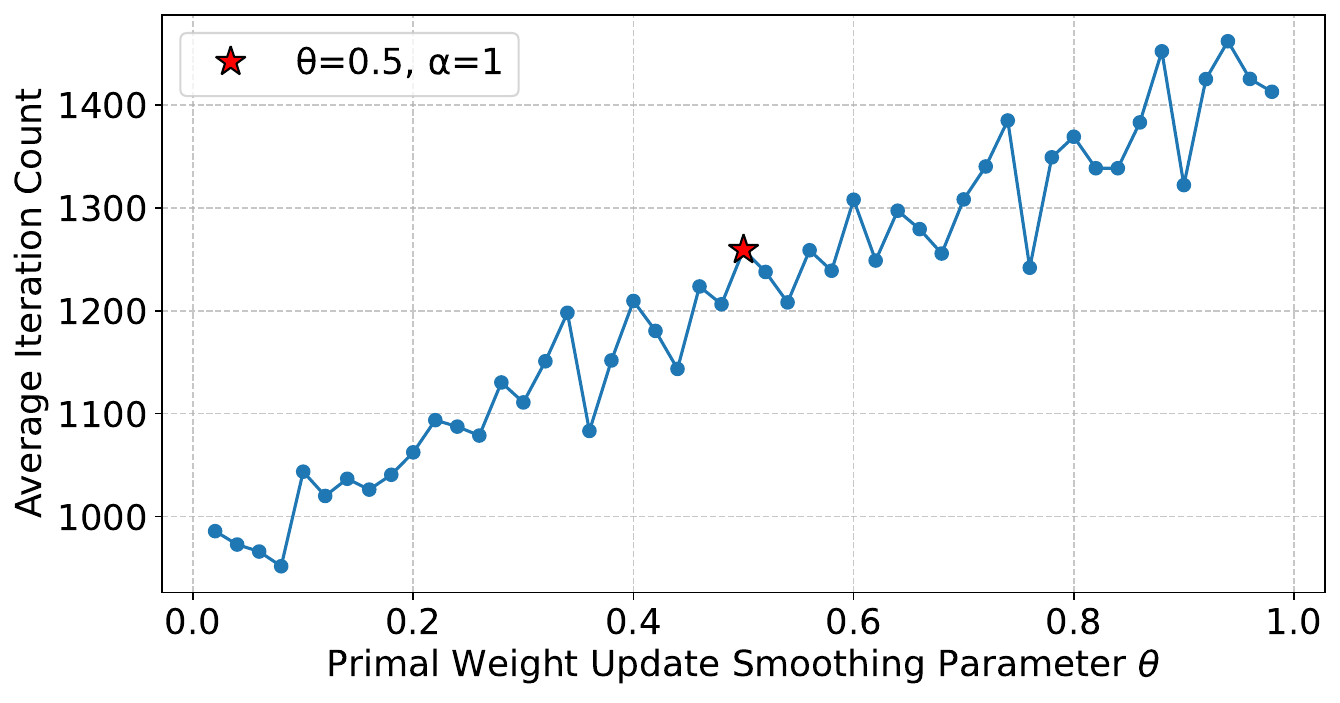}
    \includegraphics[trim={0 0 0 0},clip,width=0.495\linewidth]{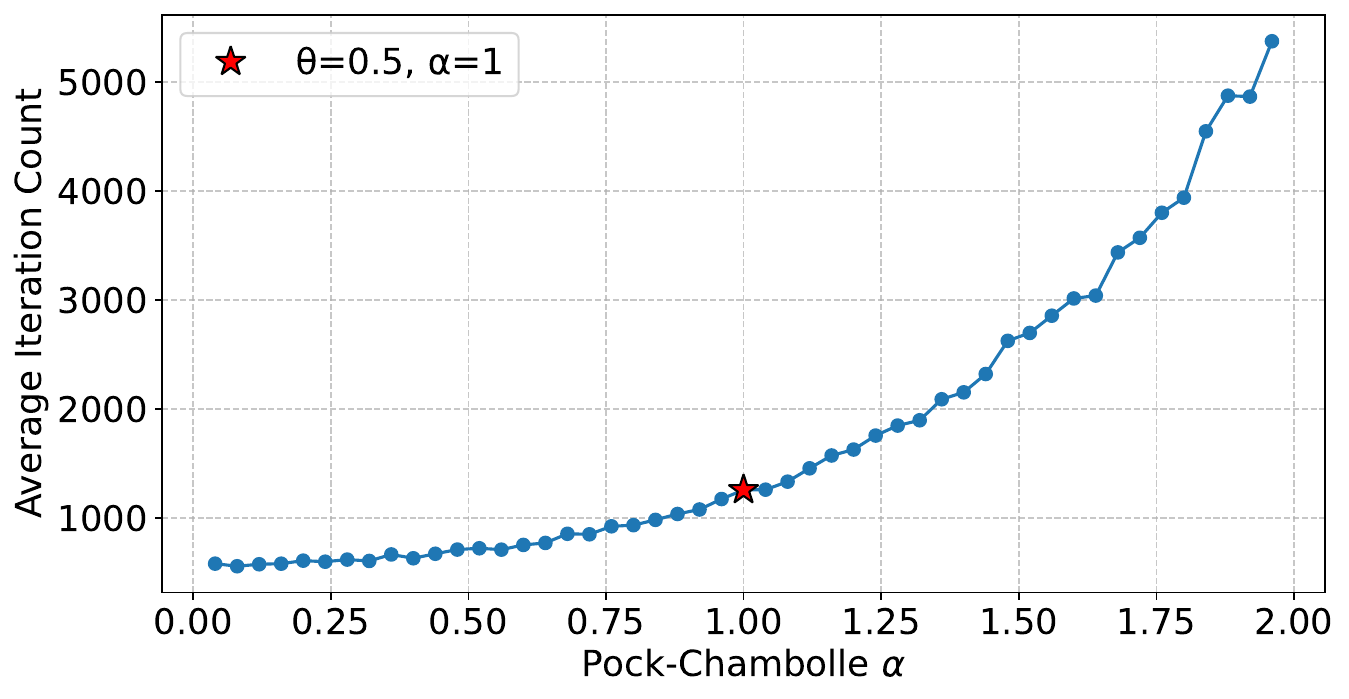}
    \includegraphics[trim={0cm 0cm 0cm 0cm},clip, width=0.495\linewidth]{figs/muca_200_2000/iterations_vs_theta.pdf}
    \includegraphics[trim={0 0 0 0},clip,width=0.495\linewidth]{figs/muca_200_2000/iterations_vs_alpha.pdf}
    \caption{Decay-decay auction iteration counts (averaged over 100 instances) with 100 items and 1000 bids (top), and 200 items and 2000 bids (bottom). {\bf (Left)} Effect of tuning $\theta$ with $\alpha = 1$. {\bf (Right)} Effect of tuning $\alpha$ with $\theta = 0.5$.}
    \label{fig:dd_auction_iterations}
\end{figure}

\begin{figure}[!htb]
    \centering
    \includegraphics[trim={0cm 0cm 0cm 0cm},clip, width=0.495\linewidth]{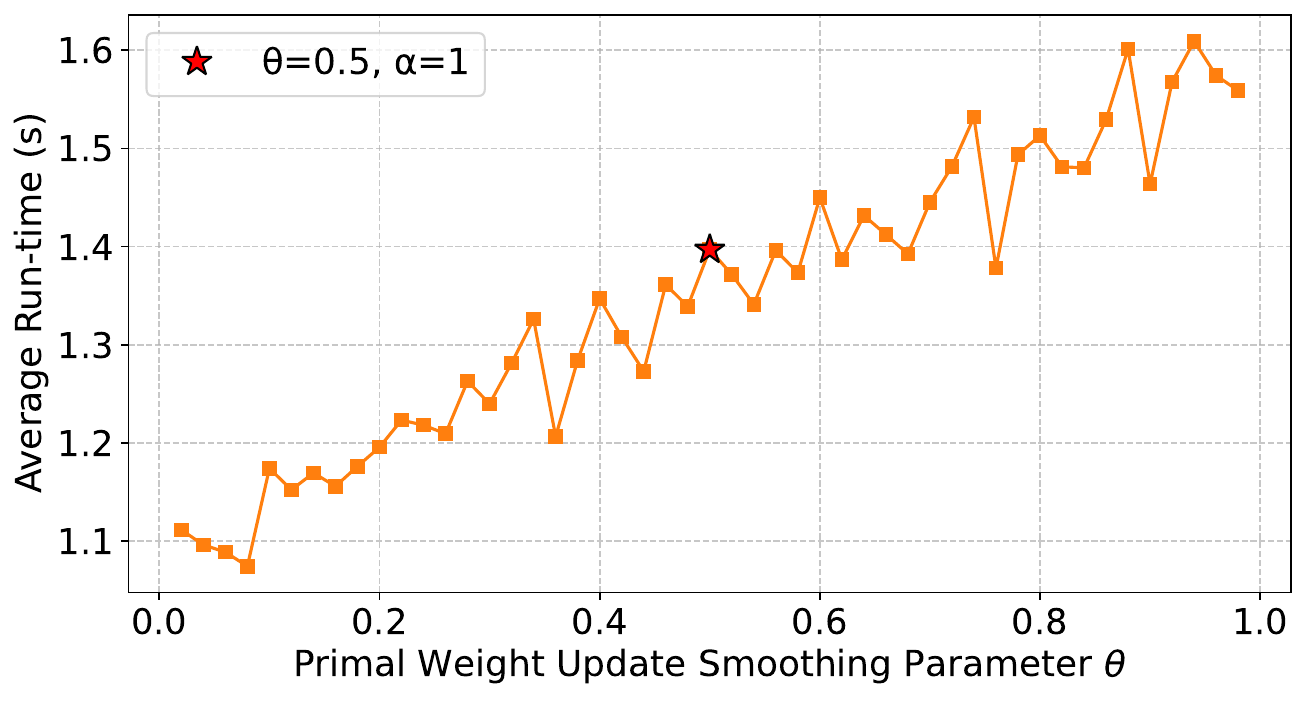}
    \includegraphics[trim={0 0 0 0},clip,width=0.495\linewidth]{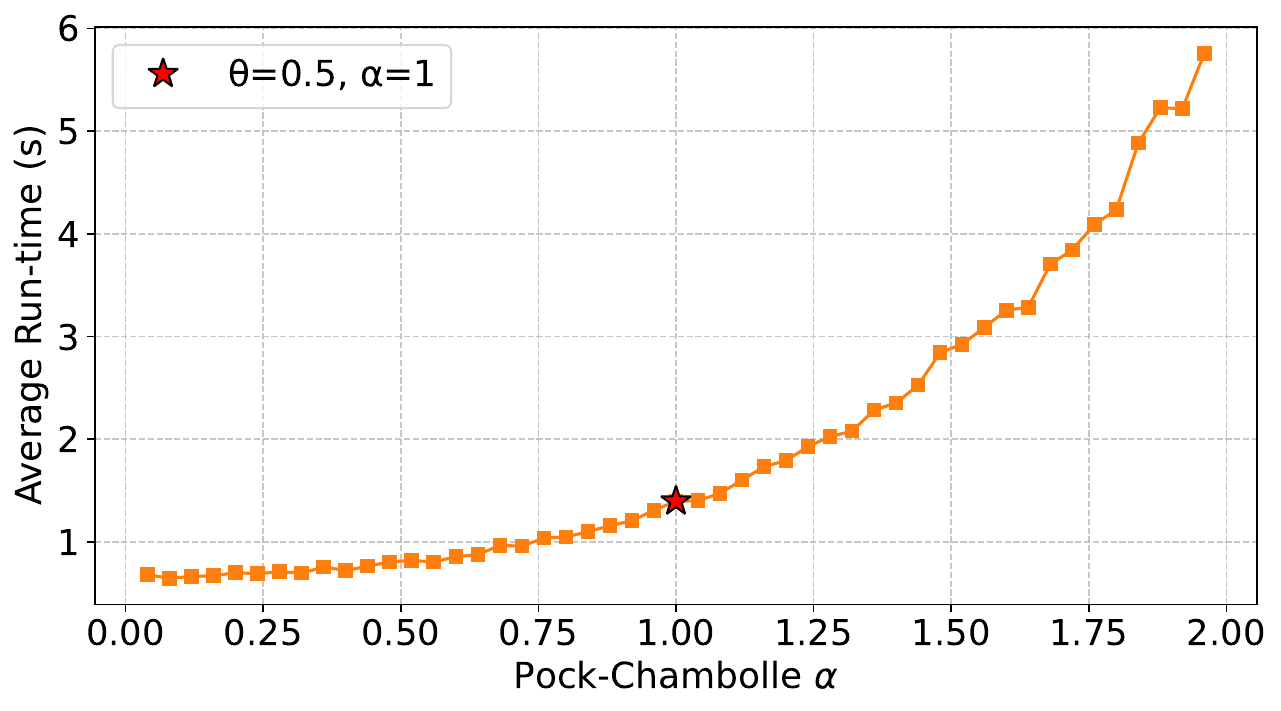}
    \includegraphics[trim={0cm 0cm 0cm 0cm},clip, width=0.495\linewidth]{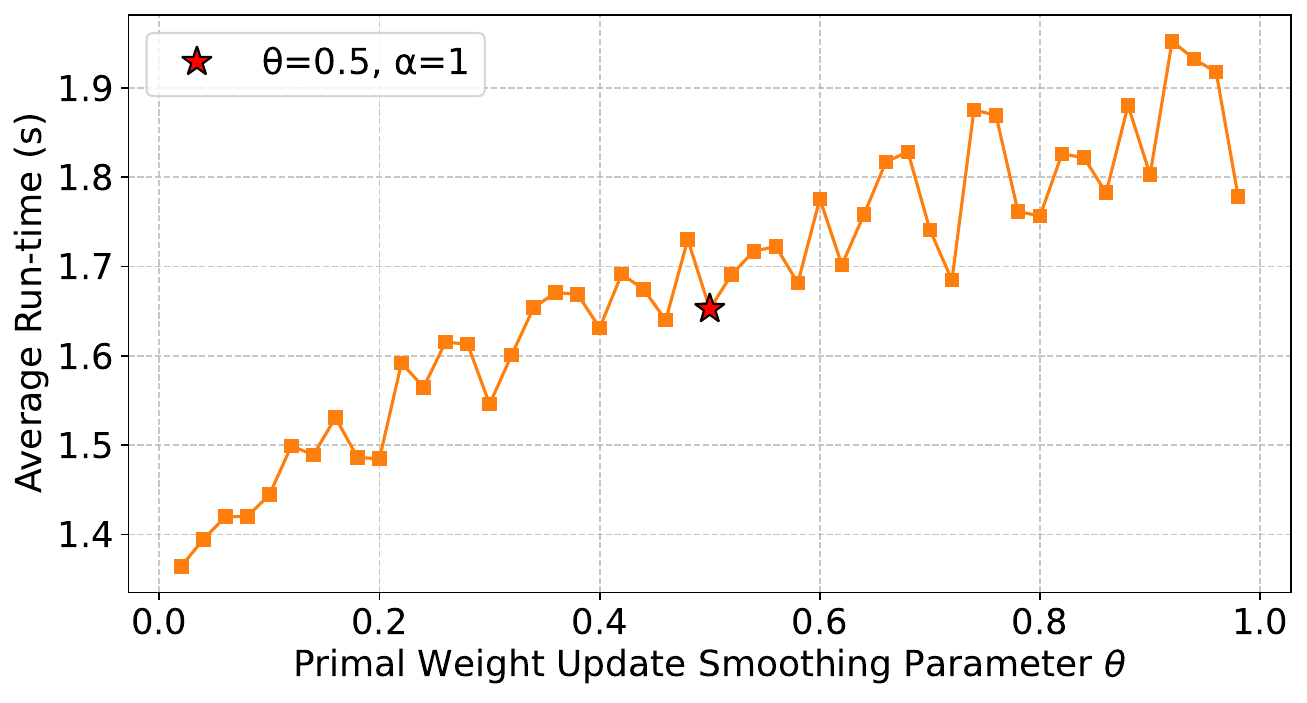}
    \includegraphics[trim={0 0 0 0},clip,width=0.495\linewidth]{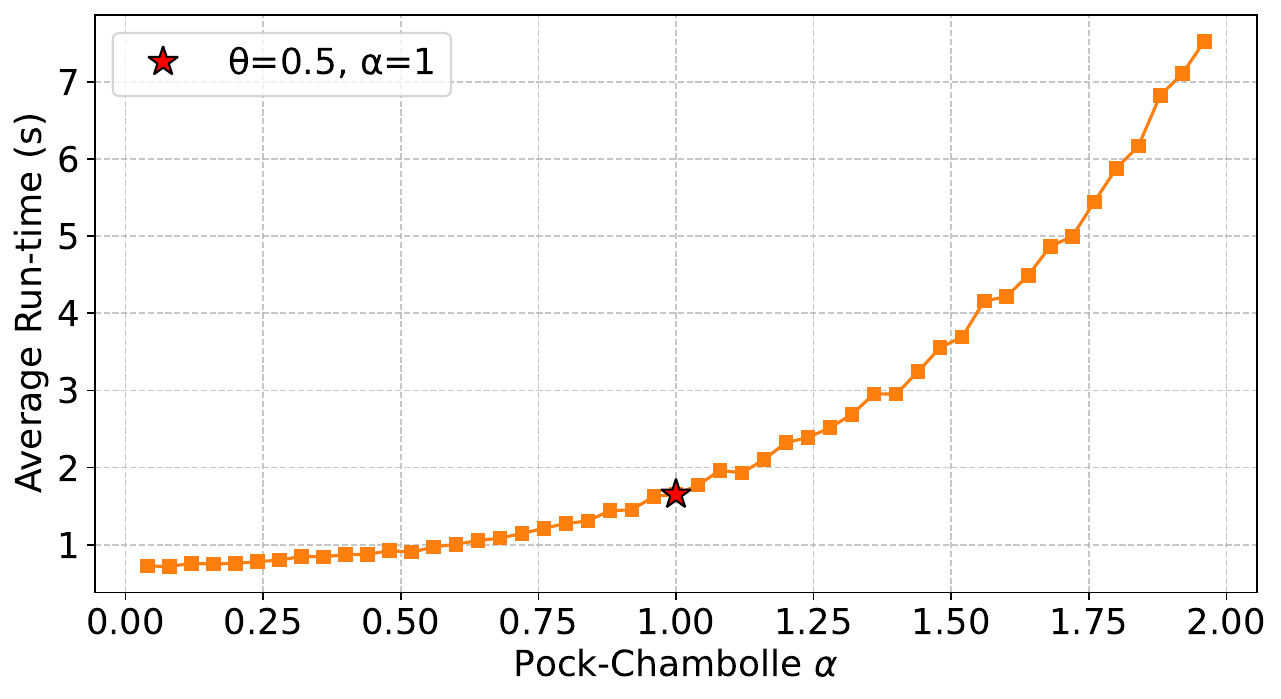}
    \caption{Decay-decay auction run-times (averaged over 100 instances) with 100 items and 1000 bids (top), and 200 items and 2000 bids (bottom). {\bf (Left)} Effect of tuning $\theta$ with $\alpha = 1$. {\bf (Right)} Effect of tuning $\alpha$ with $\theta = 0.5$.}
    \label{fig:dd_auction_runtime}
\end{figure}

\begin{figure}[!htb]
    \centering
    \includegraphics[trim={0 0 0 0},clip,width=0.495\linewidth]{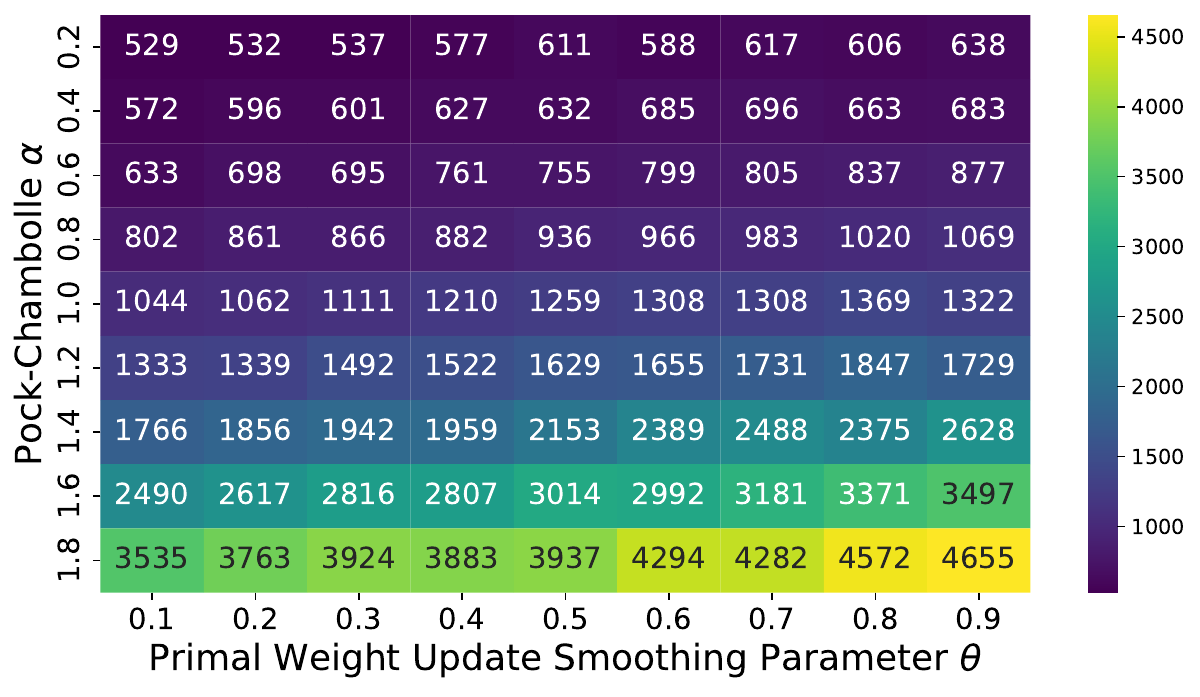}
    \includegraphics[trim={0 0 0 0},clip,width=0.495\linewidth]{figs/muca_200_2000/grid_iterations_heatmap.pdf}
    \caption{Decay-decay auction iteration counts (averaged over 100 instances); grid search over $(\theta,\alpha)$. {\bf (Left)} 100 items, 1000 bids; {\bf (Right)} 200 items, 2000 bids.}
    \label{fig:dd_auction_iterations_heatmap}
\end{figure}

\begin{figure}[!htb]
    \centering
    \includegraphics[trim={0 0 0 0},clip,width=0.495\linewidth]{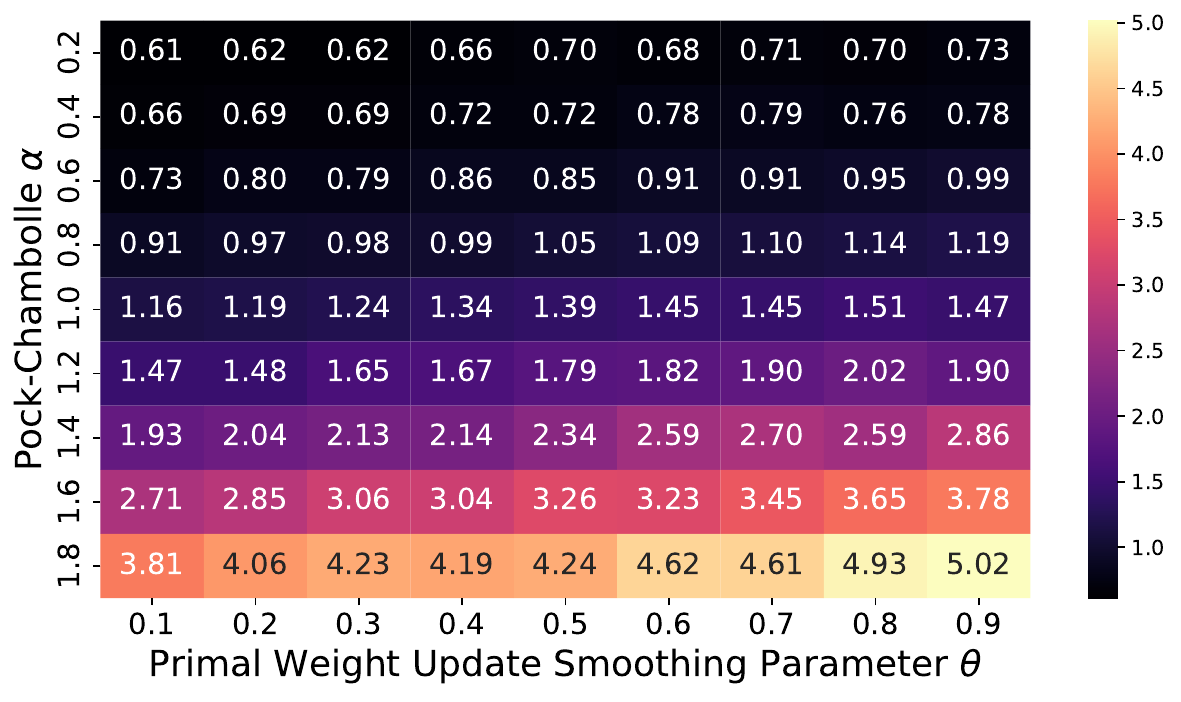}
    \includegraphics[trim={0 0 0 0},clip,width=0.495\linewidth]{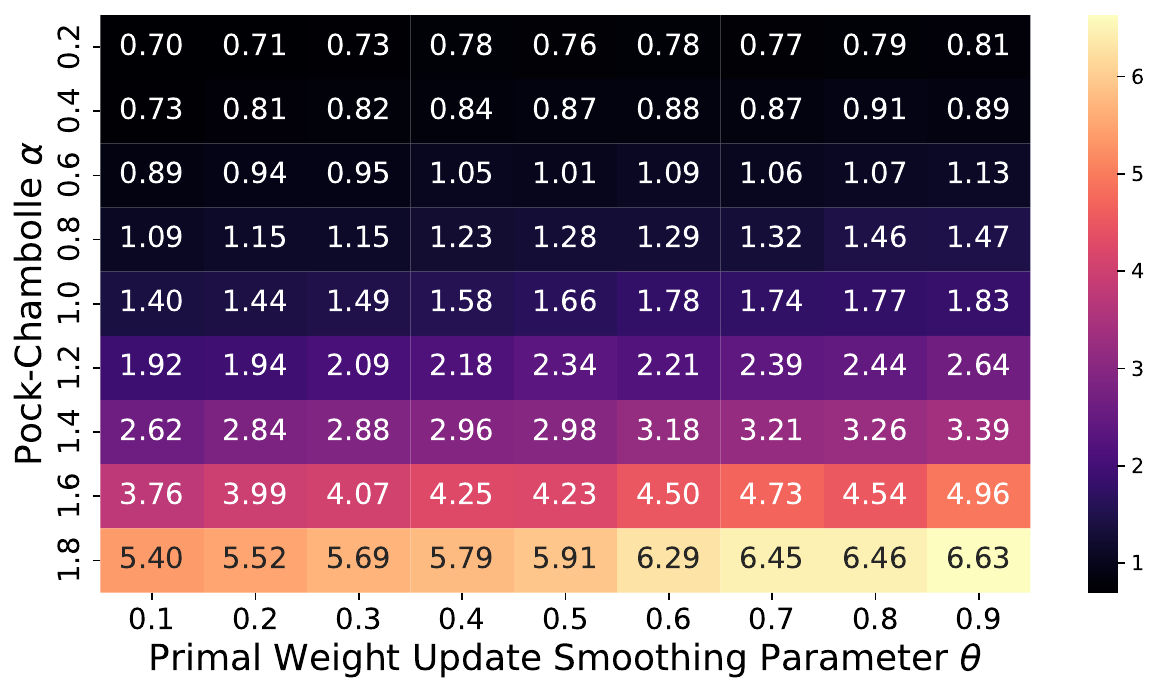}
    \caption{Decay-decay auction run-times (averaged over 100 instances); grid search over $(\theta,\alpha)$. {\bf (Left)} 100 items, 1000 bids; {\bf (Right)} 200 items, 2000 bids.}
    \label{fig:dd_auction_runtime_heatmap}
\end{figure}


\subsubsection{Multipaths instances~\citep{leyton2000towards}} Figure~\ref{fig:multipaths_iterations} displays the single-parmeter tuning iteration count results; Figure~\ref{fig:multipaths_runtime} displays the single-parameter tuning run-time results; Figure~\ref{fig:multipaths_heatmaps} displays the grid search results.

\begin{figure}[!htb]
    \centering
    \includegraphics[trim={0cm 0cm 0cm 0cm},clip, width=0.495\linewidth]{figs/multipaths_90_2000/iterations_vs_theta.pdf}
    \includegraphics[trim={0 0 0 0},clip,width=0.495\linewidth]{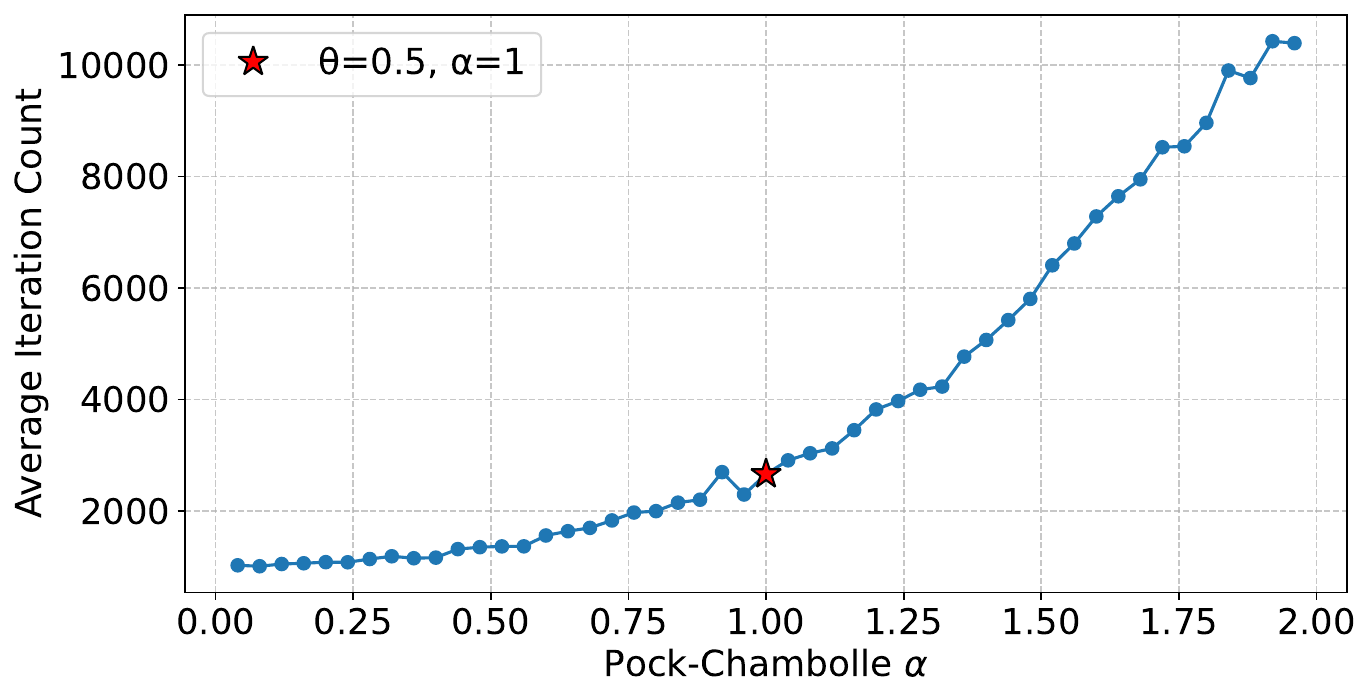}
    \caption{Multipaths auction (90 items, 2000 bids) iteration counts (averaged over 100 instances). {\bf (Left)} Effect of tuning $\theta$ with $\alpha = 1$. {\bf (Right)} Effect of tuning $\alpha$ with $\theta = 0.5$.}
    \label{fig:multipaths_iterations}
\end{figure}

\begin{figure}[!htb]
    \centering
    \includegraphics[trim={0cm 0cm 0cm 0cm},clip, width=0.495\linewidth]{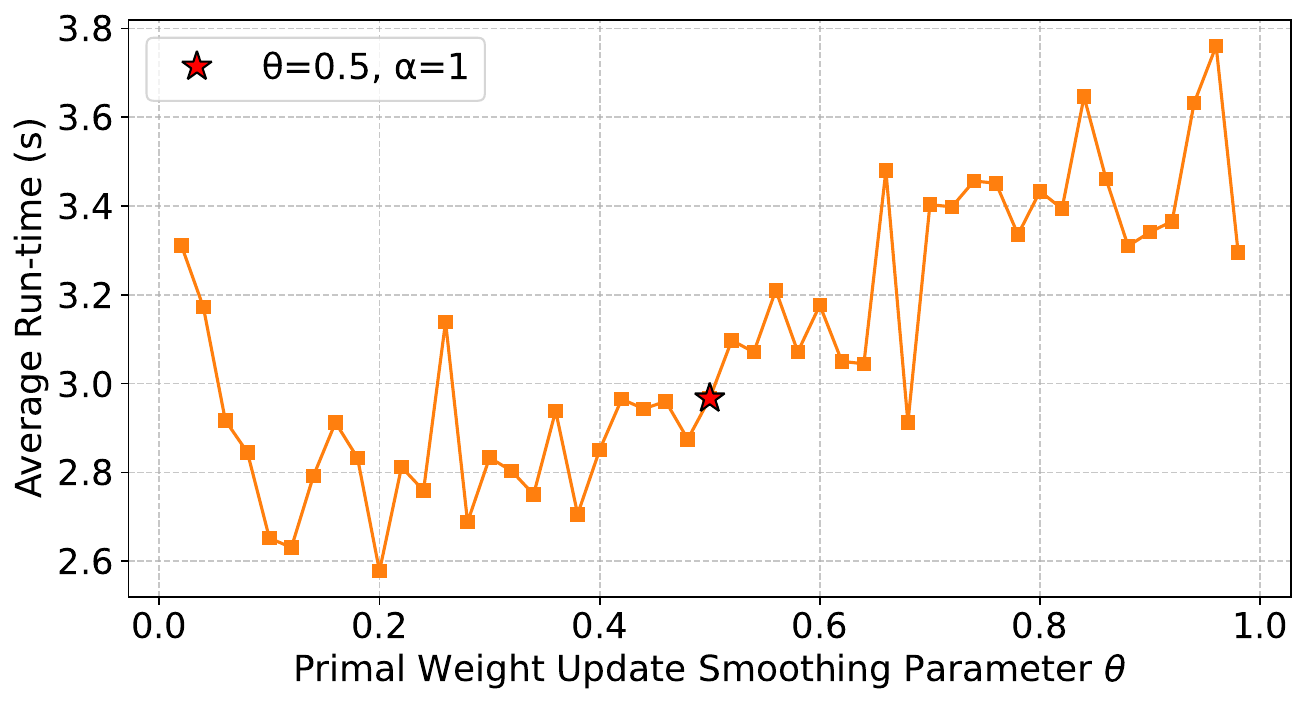}
    \includegraphics[trim={0 0 0 0},clip,width=0.495\linewidth]{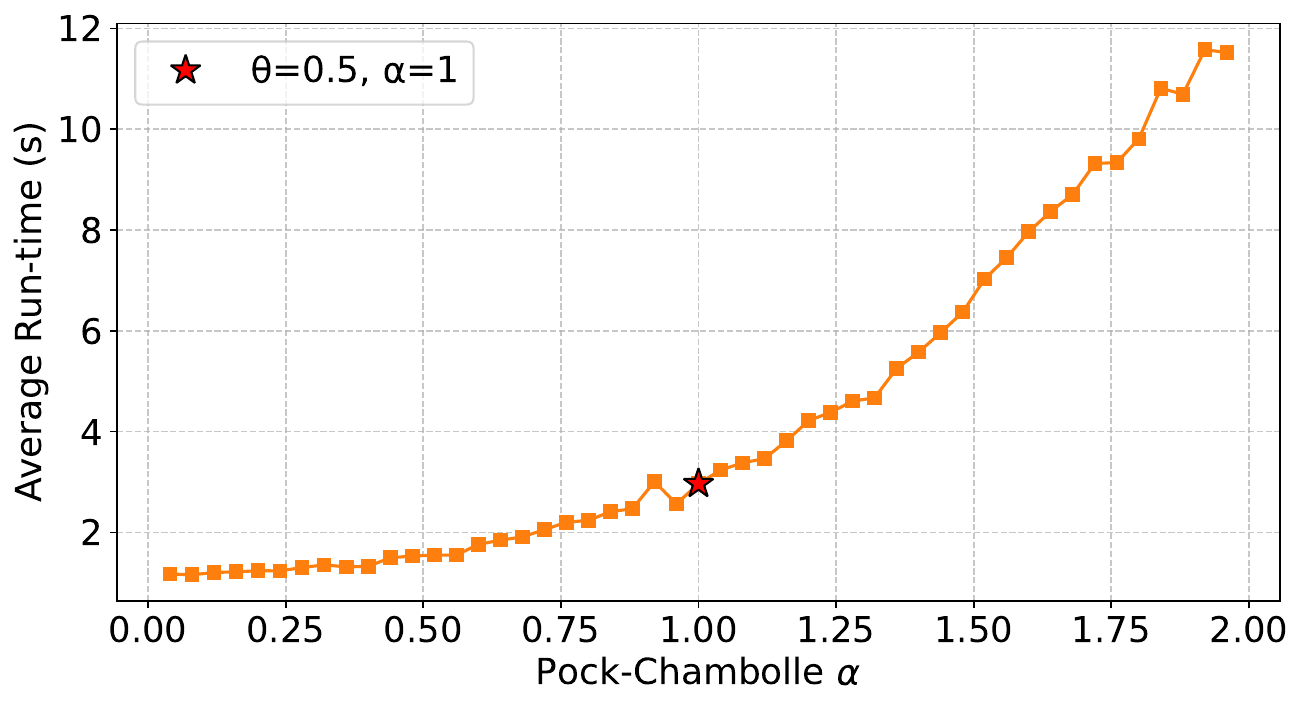}
    \caption{Multipaths auction (90 items, 2000 bids) run-times (averaged over 100 instances). {\bf (Left)} Effect of tuning $\theta$ with $\alpha = 1$. {\bf (Right)} Effect of tuning $\alpha$ with $\theta = 0.5$.}
    \label{fig:multipaths_runtime}
\end{figure}

\begin{figure}[!htb]
    \centering
    \includegraphics[trim={0 0 0 0},clip,width=0.495\linewidth]{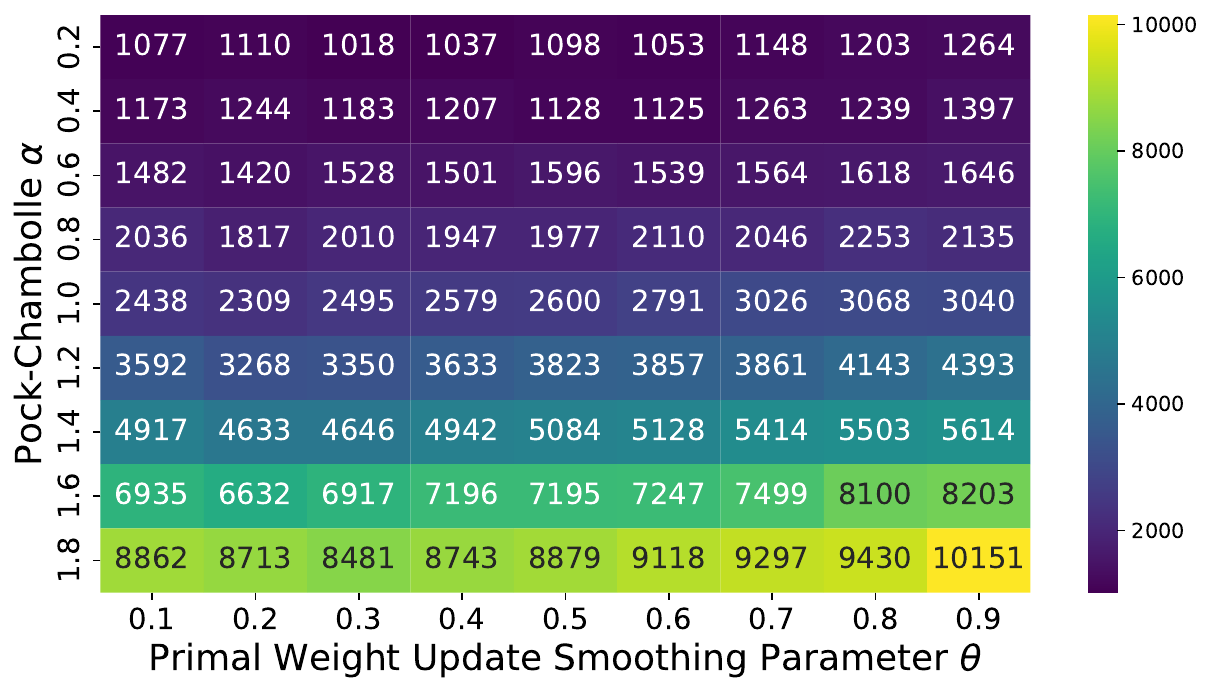}
    \includegraphics[trim={0 0 0 0},clip,width=0.495\linewidth]{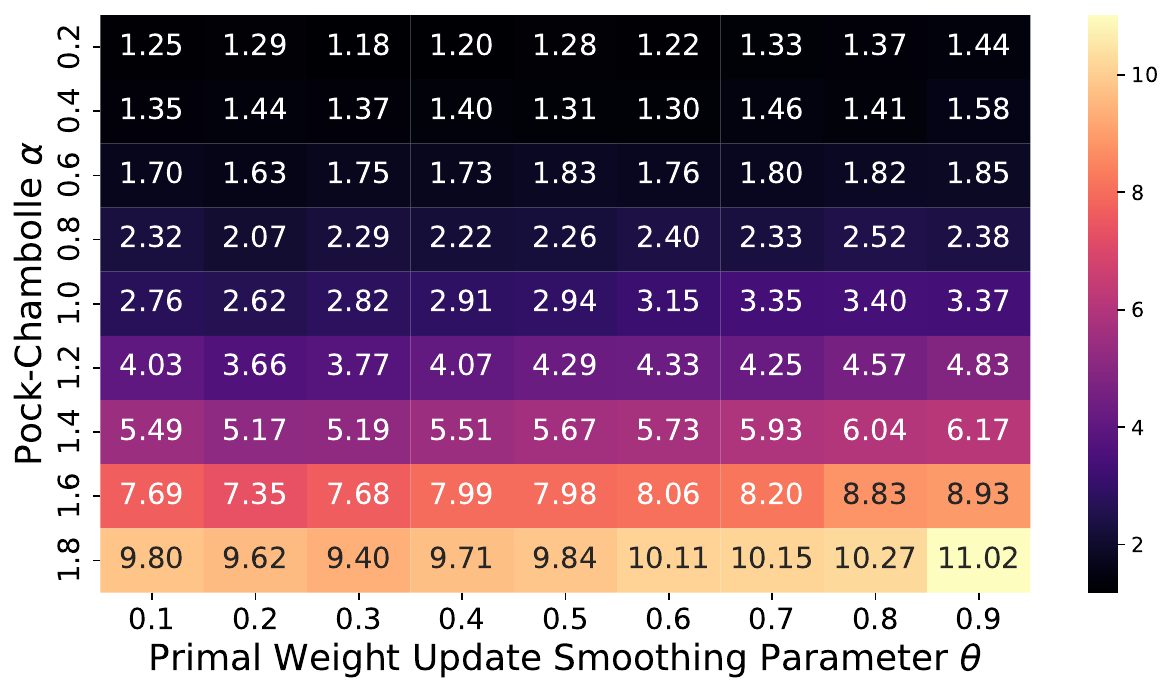}
    \caption{Multipaths auction (90 items, 2000 bids) iteration counts (left) and run-times (right); grid search over $(\theta,\alpha)$ (averaged over 100 instances).}
    \label{fig:multipaths_heatmaps}
\end{figure}

\end{document}